\documentclass[12pt,twoside]{amsart}
\usepackage{amssymb}
\usepackage{verbatim}
\usepackage{amsmath}
\usepackage{color}
\usepackage{bm}
\usepackage{a4wide}
\usepackage[latin1]{inputenc}
\usepackage[T1]{fontenc}
\usepackage{times}
\usepackage{amssymb,latexsym}
\usepackage{enumerate}
\usepackage{mathrsfs}
\usepackage{pict2e}
\usepackage{hyperref}

\makeatletter
\DeclareRobustCommand{\intprod}{%
  \mathbin{\mathpalette\int@prod{(0.1,0)(0.9,0)(0.9,0.8)}}%
}
\DeclareRobustCommand{\intprodr}{%
  \mathbin{\mathpalette\int@prod{(0.1,0.8)(0.1,0)(0.9,0)}}}

\newcommand{\int@prod}[2]{%
  \begingroup
  \sbox\z@{$\m@th#1+$}%
  \setlength\unitlength{\wd\z@}%
  \begin{picture}(1,1)
  \roundcap
  \polyline#2
  \end{picture}%
  \endgroup
}
\makeatother

\newcommand{\dbar}{\ensuremath{\overline\partial}}

\makeatletter
\newcommand{\sumprime}{\if@display\sideset{}{'}\sum%
            \else\sum'\fi}
\makeatother

\begin{document}

\numberwithin{equation}{section}

\newtheorem{theorem}{Theorem}[section]
\newtheorem{proposition}[theorem]{Proposition}
\newtheorem{conjecture}[theorem]{Conjecture}
\def\theconjecture{\unskip}
\newtheorem{corollary}[theorem]{Corollary}
\newtheorem{lemma}[theorem]{Lemma}
\newtheorem{observation}[theorem]{Observation}
\newtheorem{definition}{Definition}
\numberwithin{definition}{section} 
\newtheorem{remark}{Remark}
\def\theremark{\unskip}
\newtheorem{kl}{Key Lemma}
\def\thekl{\unskip}
\newtheorem{question}{Question}
\def\thequestion{\unskip}
\newtheorem{example}{Example}
\def\theexample{\unskip}
\newtheorem{problem}{Problem}

\thanks{The first author is supported by National Natural Science Foundation of China, No. 12271101}
\thanks{The second author is supported by Fundamental Research Fund for the Central Universities,  Southwest Minzu University, No. ZYN2023075}

\title[The $H^2-$Corona problem on delta-regular domains]{The $H^2-$Corona problem on delta-regular domains}

 \author[Bo-Yong Chen]{Bo-Yong Chen}
 \author[Xu Xing]{Xu Xing}
 \begin{abstract}
 We prove an $H^2-$Corona theorem   with estimate  $C(\delta)=C\delta^{-1-q}|\log \delta|$ for $\delta\ll 1$ on  delta-regular domains,  where $q=\min\{n,m-1\}$ and $m$ is the number of generators.  This class of domains  includes smooth bounded domains with defining functions that are plurisubharmonic on  boundaries and pseudoconvex domains of D'Angelo finite type.
  \bigskip
  
  \noindent{{\sc Mathematics Subject Classification} (2020): 32A10, 32A35, 32T25.}
  
  \smallskip
  
  \noindent{{\sc Keywords}: Hardy space, $H^2-$Corona problem, $\bar{\partial}-$operator,  delta-regular domain.} 
\end{abstract}

\address [Bo-Yong Chen] {School of Mathematical Sciences,  Fudan University,  Shanghai,  200433,  China}
\email{boychen@fudan.edu.cn}
 
\address[Xu Xing] {School of Mathematics,  Southwest Minzu University,  Chengdu,   610041,  China} 
\email{xingxu@swun.edu.cn}
\maketitle

\section{Introduction}

Let $\Omega\subset \mathbb C^n$ be a bounded pseudoconvex domain with $C^2-$boundary  and $\mathcal O(\Omega)$ the set of holomorphic functions on $\Omega$.  Let $H^\infty(\Omega)\subset \mathcal O(\Omega)$ be  the space of bounded holomorphic functions.  
Given a function space $\mathscr F\subset \mathcal O(\Omega)$,  one has the following  classical problem in complex analysis:

\begin{problem}[$\mathscr F-$Corona problem]
Given $g=(g_1,\cdots,g_m)\in H^\infty(\Omega)^{\oplus{m}}$ with 
$$
 1\geq|g|:=\sqrt{|g_1|^2+\cdots+|g_m|^2} \ge \delta>0. 
$$
 Is it possible to conclude that for each $f\in \mathscr F$ there exists $h=(h_1,\cdots,h_m)\in \mathscr F^{\oplus{m}}$ such that
 $$
\begin{cases}
g\cdot{h}:= g_1 h_1+\cdots+g_m h_m =  f & \\
\|h\|_{\mathscr F}\leq  C(\delta)\|f\|_{\mathscr F} &
\end{cases}
$$
where  $C(\delta)$ is a constant depending on $\delta$?
\end{problem}

The original Corona problem, i.e.,  the $H^\infty-$Corona problem for the unit disc  $\mathbb D$,  has been solved by Carleson \cite{Carleson} in the positive sense long time ago.  Motivated by the work of H\"ormander \cite{Hormander67},   Wolff discovered a remarkable  $\bar{\partial}-$proof of Carleson's theorem with $C(\delta)=C\delta^{-4}$ (cf. \cite{Garnnett}).  
For $\delta\ll1$,  Tolokonnikiv \cite{Tolokonnikov} got $C(\delta)=C\delta^{-2}|\log\delta|^{3/2}$ and showed that the estimate cannot be better than $C\delta^{-2}$. Actually, the estimate cannot be better than $C\delta^{-2}\log|\log \delta|$ (see Treil \cite{Treil}).  Up to now the best known result is  Uchiyama's estimate $C(\delta)=C \delta^{-2}|\log \delta|$ (cf.  \cite{Uchiyama}).  The $H^\infty-$Corona theorem (without estimate on $C(\delta)$) remains true on finitely-connected planar domains (cf.  \cite{Alling},  \cite{Stout}).  

In the context of several complex variables,   Sibony \cite{Sibony72} first constructed a bounded (non-smooth) pseudoconvex domain in $ \mathbb C^2$ on which the $H^\infty-$Corona theorem fails.  Later, the construction was modifies to produce a counterexample of smooth pseudoconvex domain in $\mathbb C^3$ (cf. Sibony \cite{Sibony87}), and in $\mathbb{C}^2$ (cf. Forn{\ae}ss-Sibony \cite{Fornaess-Sibony}).   On the other hand,  it is surprising that one can not find a domain in $\mathbb C^n$ with $n\ge 2$ on which the   $H^\infty-$Corona theorem holds,  even for simple domains like the unit ball $\mathbb B^n$.  In 1977,  Varopoulos \cite{Varopoulos} proved the ${\rm BMOA}-$Corona theorem for two generators on strongly pseudoconvex domains.  Amar \cite{Amar} proved the $H^p-$Corona theorem for  two generators on $\mathbb B^n$.   Later,  Andersson-Carlsson \cite{AC} obtained the $H^p$ and $\mathrm{BMOA}$ Corona theorems with precise estimates on strongly pseudoconvex domains.   Since the literature on the Corona problem  increases rapidly,   it is difficult to mention all of them.  For a survey up to 2014,  see  \cite{Douglas}.

The first general result on weakly pseudoconvex domains  was obtained by Andersson:  

\begin{theorem}[cf.  \cite{Andersson}, \cite{Andersson2}]
If\/ $\Omega\subset \mathbb C^n$ admits  a plurisubharmonic (psh) defining function on $\Omega$, then the $H^2-$corona problem is solvable in the positive sense with estimate $C(\delta)=C\delta^{-1-q}|\log\delta|$ \/ for $\delta\ll 1$,  where $q=\min\{n,m-1\}$ and $C$ is independent of $m$.
\end{theorem}

The goal of this paper is to generalize Andersson's theorem to a much larger class of weakly pseudoconvex domains:  

\begin{definition}[cf. \cite{ChenFu}]
A bounded domain $\Omega\subset\mathbb C^n$ is said to be delta-regular if there exist a bounded $C^2$ psh function $\lambda$ and a $C^2$ defining function $\rho$ on $\Omega$ such that $i\partial\bar{\partial}\lambda\ge \rho^{-1} i\partial\bar{\partial}\rho$.  
\end{definition}

To avoid confusion,  we use "delta-regular" instead of "$\delta-$regular" in the original paper \cite{ChenFu} since the symbol $\delta$ has been used frequently in the present paper (e.g.,  $C(\delta)$).   
It is known from \cite{ChenFu} that the class of delta-regular domains includes bounded domains with defining functions that are psh on $\partial \Omega$ and pseudoconvex domains satisfying the following condition:   
\begin{equation}\label{eq:Catlin}
\exists\,\lambda\in PSH(\Omega)\cap C^2(\Omega), \eta>0,\ \mathrm{s.t.}\ 0\leq\lambda\leq1,\ i\partial\bar{\partial} \lambda \gtrsim d^{-\eta}\, i\partial\bar{\partial} |z|^2,\ \ \ \forall\,z\in\Omega,
\end{equation}
where $d$ is the Euclidean boundary distance and $PSH(\Omega)$ is the set of plurisubharmonic (psh) functions on $\Omega$. Catlin \cite{Catlin87} obtained subelliptic estimates for the $\bar\partial-$Neumann problem under \eqref{eq:Catlin}. He also showed that \eqref{eq:Catlin} holds on pseudoconvex domains of  finite type in the sense of D'Angelo \cite{DAngelo}. In \S\,\ref{subsec:delta_regular_log}, we obtain delta-regularity under the following much weaker condition:
 \begin{equation}\label{eq:logGrowth_0}
\exists\,\lambda\in PSH(\Omega)\cap C^2(\Omega), \alpha>1,\ \mathrm{s.t.}\ 0\leq\lambda\leq1,\ i\partial\bar{\partial} \lambda \gtrsim |\log d|^{\alpha}\, i\partial\bar{\partial} |z|^2,\ \ \ \forall\,z\in\Omega.
 \end{equation}
Note that \eqref{eq:logGrowth_0} is also used in Kohn's theory of superlogarithm estimates for the $\bar\partial-$Neumann problem (cf. \cite{Kohn}; see also \cite{KZ}). It remains an open problem whether $\Omega$ is delta-regular if $\partial\Omega$ verifies Catlin's property (P), i.e.,
\[
\forall\,M>0,\ \exists\,\lambda\in PSH(\Omega)\cap C^\infty(\overline{\Omega}),\ \mathrm{s.t.}\ 0\leq\lambda\leq1,\ i\partial\bar{\partial} \lambda \ge M i\partial\bar{\partial} |z|^2,\ \ \ \forall\,z\in \partial \Omega.
\]
Catlin \cite{Catlin84} obtained global regularity of the $\bar\partial-$Neumann problem under property (P).

Our main result is given as follows. 

\begin{theorem}\label{th:Main}
If\/ $\Omega\subset \mathbb C^n$ is a delta-regular domain, then the $H^2-$corona problem is solvable in the positive sense with estimate $C(\delta)=C\delta^{-1-q}|\log\delta|$ \/ for $\delta\ll 1$,  where $q=\min\{n,m-1\}$ and $C$ is independent of $m$. 
\end{theorem}

In our proof of Theorem \ref{th:Main}, we  apply the way of solving the $\bar{\partial}-$equation in H\"ormander's 1965 Acta paper \cite{Hormander}  instead of solving the $\bar{\partial}_b-$equation in \cite{Andersson}, where Kohn's regularity theory on the $\bar{\partial}-$Neumann problem has to be used.  It seems that  the three-weights technique in H\"ormander's book \cite{HormanderBook} does not work (see \S\,5 for a brief discussion).  This is exactly the primary motivation of writing this paper.  In order to make this comparison more transparent,  we shall treat the special case of two generators (without estimate on $C(\delta)$) separately,  since  it suffices to solve a {\it single} $\bar{\partial}-$equation (which is more simple) and the main\/ {\it analytic}\/ idea is essentially contained in this case.   The general case  is more complicated\/ {\it algebraically}: we have to solve a vector-valued $\bar{\partial}-$equation with certain\/ {\it additional} condition and some techniques due to Skoda \cite{Skoda} (and simplified by Ohsawa \cite{Ohsawa}) are necessary.  In order to get the estimate on $C(\delta)$,    more delicate analysis is needed.   Finally,  we remark that  the crucial inequality that we have used,  goes back to Berndtsson \cite{Berndtsson},  which is also contained in \cite{Andersson}. 

A natural question arises

\begin{problem}
Does the $H^2-$Corona problem have a solution on any bounded pseudoconvex domain with $C^2-$boundary?   
\end{problem}

It is shown in \cite{Chen14} that the $A^2_\alpha(\Omega)-$Corona problem (for two generators) is solvable in the positive sense on any bounded pseudoconvex domain with $C^2-$boundary for each $\alpha>-1$,  where 
$$
A^2_\alpha(\Omega) = \left\{f\in \mathcal O(\Omega):  \int_\Omega |f|^2 d^\alpha <\infty\right\}.  
$$
Note that the Hardy space $H^2(\Omega)$ may be viewed as the limit of the (weighted) Bergman space $A^2_\alpha(\Omega)$ as $\alpha\rightarrow -1+$. 

\section{Preliminaries}
\subsection{Carleson-type inequalities}
Let  $\Omega$ be a bounded domain with $C^2-$boundary in $\mathbb C^n$ and  $\rho$ a $C^2$ defining function. We have 

\begin{lemma}\label{lm:sh}
 There exists a $C^2$ subharmonic defining function ${\varrho}$ such that ${\varrho} =\rho+A\rho^2$ in a neighborhood of $\partial \Omega$\/ for some $A>0$.  
\end{lemma}

\begin{proof}
Set $\tilde{\rho}:=\rho+A\rho^2$. Since 
$$
\Delta \tilde{\rho} = \Delta \rho + 2 A\rho\Delta \rho +2A |\nabla \rho|^2, 
$$
we conclude that $\tilde{\rho}$ is strictly subharmonic  in a neighborhood $U$ of $\partial \Omega$ whenever $A$ is sufficiently large.  We may choose $r>0$ such that  $\{z\in \overline{\Omega}\cap U: \tilde{\rho}(z)\ge -r\}$ is relatively compact in $\overline{\Omega}\cap U$. We also choose a smooth convex increasing function $\kappa$ with $\kappa=-3r/4$ on $(-\infty,-r]$ and $\kappa(t)=t$ on $[-r/2,\infty)$, and $\chi\in C^\infty_0(\Omega)$ with $\chi=1$ on $\{\kappa(\tilde{\rho})\le -r/2\}$. It suffices to take
$$
\varrho=\kappa(\tilde{\rho})+c\chi\cdot |z|^2
$$
for suitable small constant $c>0$.
\end{proof}

Set $\Omega_\varepsilon=\{z\in \Omega:\rho_\varepsilon(z):=\rho(z)+\varepsilon<0\}$ for $\varepsilon>0$. We have the following Carleson-type inequalities: 

\begin{lemma}[compare \cite{AnderssonBook}]\label{lm:Green}
Let $\psi$ be a bounded $C^2$ subharmonic function on $\Omega$ and $f\in \mathcal O(\Omega)$. Then there exists a positive constant $C$ depending only on $\Omega$  such that for all $\varepsilon\ll1$,
\begin{equation}\label{eq:G_1}
\int_{\Omega_\varepsilon} (-\rho_\varepsilon) |f|^2 \Delta \psi \le C\|\psi\|_{L^\infty(\Omega)}\int_{\partial \Omega_\varepsilon} |f|^2  d\sigma_\varepsilon,
\end{equation}
\begin{equation}\label{eq:G_2}
\int_{\Omega_\varepsilon} (-\rho_\varepsilon) |f|^2 |\nabla \psi|^2  \le C\|\psi\|_{L^\infty(\Omega)}^2\int_{\partial \Omega_\varepsilon} |f|^2  d\sigma_\varepsilon,
\end{equation}
\begin{equation}\label{eq:G_3}
\int_{\Omega_\varepsilon} (-\rho_\varepsilon) |\partial f|^2   \le C\int_{\partial \Omega_\varepsilon} |f|^2  d\sigma_\varepsilon.
\end{equation}
Here $d\sigma_\varepsilon$ denotes the surface element on $\partial \Omega_\varepsilon$.
\end{lemma}

\begin{proof}
Let $\varrho$ be chosen as Lemma \ref{lm:sh}. Set $\varrho_\varepsilon=\varrho+\varepsilon-A\varepsilon^2$. Since 
$$
-\varrho_\varepsilon=(-\rho_\varepsilon)\left[1+A(\rho-\varepsilon)\right]
$$
holds near $\partial\Omega$, it follows that $-\varrho_\varepsilon\asymp -\rho_\varepsilon$ for $\varepsilon\ll1$ with implicit constants depending only on $\Omega$.
 Thus it suffices to verify $\eqref{eq:G_1}\sim \eqref{eq:G_3}$ with $\rho_\varepsilon$ replaced $\varrho_\varepsilon$. By Green's formula, we have
$$
 \int_{\Omega_\varepsilon} (-\varrho_\varepsilon)  (\Delta \psi+|\nabla \psi|^2) e^\psi
  =  \int_{\Omega_\varepsilon} (-\varrho_\varepsilon)\Delta e^\psi
 =  -\int_{\Omega_\varepsilon} \Delta \varrho_\varepsilon\,e^\psi + \int_{\partial \Omega_\varepsilon} \frac{\partial\varrho_\varepsilon}{\partial \nu_\varepsilon}\,e^\psi d\sigma_\varepsilon
  \le  C \int_{\partial \Omega_\varepsilon} e^\psi d\sigma_\varepsilon
$$
 because $\Delta \varrho_\varepsilon\ge 0$, so that 
 \begin{equation}\label{eq:G_5}
 \int_{\Omega_\varepsilon} (-\varrho_\varepsilon)  \Delta \psi\, e^\psi 
 \le C \int_{\partial \Omega_\varepsilon} e^\psi d\sigma_\varepsilon,
 \end{equation}
  \begin{equation}\label{eq:G_6}
 \int_{\Omega_\varepsilon} (-\varrho_\varepsilon)  |\nabla \psi|^2 e^\psi 
 \le C \int_{\partial \Omega_\varepsilon} e^\psi d\sigma_\varepsilon.
 \end{equation}   
Apply \eqref{eq:G_6} with $\psi=\log (|f|^2+\tau)$ and then let $\tau\downarrow 0$, we obtain \eqref{eq:G_3}. Analogously, apply \eqref{eq:G_5} with $\psi$ replaced by $\psi/\|\psi\|_{L^\infty(\Omega)}+\log (|f|^2+\tau)$ and then let $\tau\downarrow 0$, we obtain \eqref{eq:G_1}. Finally, apply \eqref{eq:G_6} with $\psi$ replaced by $\psi/\|\psi\|_{L^\infty(\Omega)}+\log (|f|^2+\tau)$ and then let $\tau\downarrow 0$, we obtain 
 $$
  \int_{\Omega_\varepsilon} (-\varrho_\varepsilon)  \left|\nabla \psi/\|\psi\|_{L^\infty(\Omega)}+\nabla\log |f|^2\right|^2\, |f|^2 
 \le C \int_{\partial \Omega_\varepsilon} |f|^2 d\sigma_\varepsilon.
 $$
 This combined with \eqref{eq:G_3} gives \eqref{eq:G_2}.
  \end{proof}
 
 \subsection{Hardy space}  
  Let $\Omega$ be a bounded domain with $C^2-$boundary in $\mathbb C^n$.  Given a $C^2$ defining function $\rho$,  set $\Omega_\varepsilon:=\{z\in \Omega: \rho(z)<-\varepsilon\}$.  Following Stein \cite{Stein},  we define the harmonic Hardy space $h^p(\Omega)$,  where $1<p<\infty$,  to be the space of harmonic functions $f$ satisfying
  $$
  \|f\|_{h^p(\Omega)}^p:=\limsup_{\varepsilon\rightarrow 0+} \int_{\partial \Omega_\varepsilon} |f|^p d\sigma_\varepsilon <\infty.
  $$
This definition is independent of the choice of defining functions and the non-tangential limit $f^\ast$ of $f$ exists almost everywhere on $\partial \Omega$.  Furthermore,  $f^\ast \in L^p(\partial \Omega)$ and $\|f\|_{h^p(\Omega)}=\|f^\ast\|_{L^p(\partial \Omega)}$.  It is also known that if $f\in h^p(\Omega)$ then
  $$
  \|f\|_{h^p(\Omega)}^p = \lim_{r\rightarrow 1-} (1-r) \int_\Omega |f|^p d^{-r}
  $$
  where $d=d(\cdot,\partial \Omega)$  (cf.  \cite{ChenFu},  Lemma 3.2).  
  
  The (holomorphic) Hardy space $H^p(\Omega)$ is defined to be $h^p(\Omega)\cap \mathcal O(\Omega)$.  The Hilbert space $H^2(\Omega)$ is of particular interest,  since it admits a reproducing kernel,  i.e.,  the Szeg\"o kernel $S_\Omega(z,w)$,  which is of fundamental importance.  
  
  \subsection{A sufficient condition for delta-regularity}
\label{subsec:delta_regular_log}

 \begin{proposition}\label{prop:delta-regular}
 Let $\Omega\subset \mathbb C^n$ be a bounded pseudoconvex domain with $C^2-$boundary. If \eqref{eq:logGrowth_0} holds, then $\Omega$ is delta-regular.   
 \end{proposition}
 
 \begin{proof}
  First of all,  it is known from \cite{DF1} that there exists a number $\eta>0$ and a $C^2$ defining function $ \rho\ge -1/2$ such that 
 \begin{equation}\label{eq:DF}
 i\partial \bar{\partial} \left[ -(-\rho)^\eta \right] \gtrsim \eta |\rho|^\eta \left(  i\partial \bar{\partial} |z|^2+|\rho|^{-2} i\partial \rho\wedge \bar{\partial} \rho \right).  
 \end{equation}
 Set $\psi:=-(-\rho)^\eta $ and $\phi:=\kappa(-\log(-\psi))$,  where $\kappa(t)=-t^{1-\alpha}$,  $t\ge 1$.  Then we have
 \begin{eqnarray*}
  i\partial \bar{\partial} \phi & = & \kappa'(-\log(-\psi))   i\partial \bar{\partial}(-\log(-\psi)) + \kappa''(-\log(-\psi)) \frac{i\partial \psi\wedge \bar{\partial}\psi}{\psi^2}\\
  & = & \kappa'(-\log(-\psi))   \frac{i\partial \bar{\partial} \psi}{-\psi}+ \left[ \kappa'(-\log(-\psi)) +\kappa''(-\log(-\psi)) \right] \frac{i\partial \psi\wedge \bar{\partial}\psi}{\psi^2}\\
  & \gtrsim & \frac{ i\partial \bar{\partial} |z|^2} {|\log (-\rho)|^\alpha} + \frac{i\partial \rho\wedge \bar{\partial} \rho}{\rho^2 |\log (-\rho)|^\alpha}
 \end{eqnarray*}
 near $\partial \Omega$ in view of \eqref{eq:DF}.

 Now define
$$
N=|\partial \rho|^{-1} \sum_j \frac{\partial \rho}{\partial \bar{z}_j}\frac{\partial}{\partial z_j}
$$
near $\partial \Omega$.  For any complex tangent vector $X$ we may write $X=X_N+X_T$ where $X_N=\langle X,N\rangle N$.   Since $\Omega$ is pseudoconvex,  it follows that 
 \begin{eqnarray*}
 \rho^{-1}  i\partial \bar{\partial}\rho(z;X) & \lesssim &  |X|^2 + |\rho|^{-1} |X|\,|X_N|\\
 & \lesssim & |\log(-\rho)|^\alpha |X|^2 + \frac{|X_N|^2}{\rho^2 |\log(-\rho)|^\alpha}
 \end{eqnarray*}
 holds near $\partial \Omega$.  
 Thus if we take $\tilde{\lambda}=C(\lambda+\phi)$ with $C\gg 1$ then 
 $$
  i\partial \bar{\partial} \tilde{\lambda} \ge \rho^{-1}  i\partial \bar{\partial}\rho
 $$
 holds near $\partial \Omega$. The above inequality holds on whole $\Omega$ if $\tilde{\lambda}$ is replaced by $\tilde{\lambda}+M|z|^2$ for $M\gg1$.
 \end{proof}

\section{The case of two generators}

  We shall follow H\"ormander's classical paper \cite{Hormander}.  Let $U$ be a bounded pseudoconvex domain with $C^2-$boundary.  Given a weight function $\varphi\in C^2(\overline{U})$,   define $L^2_{(p,q)}(U,\varphi)$ to be the Hilbert space of $(p,q)-$forms 
$$
\omega={\sum_{|I|=p}}' {\sum_{|J|=q}}' \omega_{I{J}} dz_I\wedge d\bar{z}_J\ \ \ \text{where}\ \ \ \int_U |\omega_{I{J}}|^2 e^{-\varphi}<\infty,\ \forall\,I,J.
$$
 The $\bar{\partial}-$operator 
induces two closed and densely defined  operators
$$
T: L^2(U,\varphi)\rightarrow L^2_{(0,1)}(U,\varphi),\ \ \  S: L^2_{(0,1)}(U,\varphi)\rightarrow L^2_{(0,2)}(U,\varphi).
$$  
Set
$$
D_{\mathcal N}(U)=\left\{\omega={\sum}_j \omega_j d\bar{z}_j \in C^1_{(0,1)}(\overline{U}): {\sum}_j \omega_j {\partial \rho}/{\partial z_j}=0\ \text{on\ }\partial{U}\right\}.
$$
Let $T^\ast_\varphi$ denote the adjoint of $T$ with respect to the inner product $(\cdot,\cdot)_\varphi$.
It is known that  $D_{\mathcal N}(U)$ lies dense in ${\rm Dom}(T^\ast_\varphi)\cap {\rm Dom}(S)$ with respect to the graph norm
$$
\omega\rightarrow \|\omega\|_{\varphi}+\|T^\ast_\varphi \omega\|_{\varphi}+\|S \omega\|_{\varphi}.
$$
Since $U$ is pseudoconvex, we have the fundamental   Morrey-Kohn-H\"ormander inequality  
\begin{eqnarray}\label{eq:MKH}
 \int_U |S \omega|^2 e^{-\varphi}+\int_U |T^\ast_\varphi \omega|^2 e^{-\varphi}
& \ge &  \int_U \sum_{j,k} \frac{\partial^2\varphi}{\partial z_j\partial\bar{z}_k} \omega_j\bar{\omega}_ke^{-\varphi}
+ \int_U \sum_{j,k}\left|\frac{\partial \omega_j}{\partial \bar{z}_k}\right|^2 e^{-\varphi}
\end{eqnarray}
 for all $\omega\in D_{\mathcal N}(U)$.

Now let $\Omega$ be as Theorem \ref{th:Main} and suppose $\varphi\in C^2(\Omega)$.  
For each $\varepsilon>0$,  we set 
$$
\Omega_\varepsilon=\{d>\varepsilon\}.
$$
Clearly,  $\Omega_\varepsilon$ is a $C^2$ pseudoconvex domain for $\varepsilon\ll1$.   
 For  $\phi:=-r\log (-\rho)$ where $0<r<1$,  we have
$$
i\partial\bar{\partial} \phi = \frac{r i\partial\bar{\partial}\rho}{-\rho} + \frac{i\partial \phi\wedge \bar{\partial}\phi} r.
$$
On the other hand,  the Cauchy-Schwarz inequality gives
$$
\left|T^\ast_{\varphi+\phi} \omega\right|^2 \le (1-r)^{-1} \left|T^\ast_\varphi \omega\right|^2 + r^{-1}\left|\sum_j \frac{\partial \phi}{\partial z_j}\,\omega_j \right|^2.
$$
Thus if we  apply (\ref{eq:MKH}) with $\Omega$ and $\varphi$ replaced by $\Omega_\varepsilon$ and $\varphi+\phi$ respectively, then we obtain
\begin{eqnarray}\label{eq:Berndtsson}
&& \int_{\Omega_\varepsilon} |S \omega |^2 (-\rho)^{r} e^{-\varphi} + (1-r)^{-1}
 \int_{\Omega_\varepsilon} |T^\ast_\varphi  \omega|^2 (-\rho)^{r} e^{-\varphi}\\
 & \ge & \int_{\Omega_\varepsilon} \sum_{j,k} \left[\frac{\partial^2\varphi}{\partial z_j\partial\bar{z}_k}+r (-\rho)^{-1}
 \frac{\partial^2\rho}{\partial z_j\partial\bar{z}_k}\right] \omega_j\bar{\omega}_k (-\rho)^{r} e^{-\varphi} \nonumber\\
 && + \int_{\Omega_\varepsilon} \sum_{j,k}\left|\frac{\partial \omega_j}{\partial \bar{z}_k}\right|^2 (-\rho)^{r} e^{-\varphi},\ \ \ \forall\,\omega\in D_{\mathcal N}(\Omega_\varepsilon).\nonumber
  \end{eqnarray}
  This crucial inequality actually goes back to Berndtsson  \cite{Berndtsson}, which turns out to be a special case of the twisted Morrey-Kohn-H\"ormander inequality due to Ohsawa-Takegoshi \cite{OT}. 
    
  Set $g=(g_1,g_2)$, $|g|^2=|g_1|^2+|g_2|^2$ and $|\partial g|^2=|\partial g_1|^2+|\partial g_2|^2$. Clearly, we have
  $$
  |\partial g|^2 = \frac{\Delta |g|^2}4.
  $$
 Set 
  $
  \varphi(z)=\lambda(z)+|z|^2 + \log |g(z)|^2.
  $
  Since  $\lambda$ is psh and satisfies $i\partial\bar{\partial}\lambda\ge \rho^{-1} i\partial\bar{\partial}\rho$, it follows that
  \begin{eqnarray*}
  i\partial\bar{\partial} \varphi + r(-\rho)^{-1} i\partial\bar{\partial} \rho
  & = & i\partial\bar{\partial} |z|^2 +i\partial\bar{\partial} \log |g|^2 + \frac{(-\rho) i\partial\bar{\partial}\lambda + r i\partial\bar{\partial}\rho}{-\rho} \\
  & \ge & i\partial\bar{\partial} |z|^2 +i\partial\bar{\partial} \log |g|^2 + \frac{r(-\rho) i\partial\bar{\partial}\lambda + r i\partial\bar{\partial}\rho}{-\rho} \\
  & \ge & i\partial\bar{\partial} |z|^2 +i\partial\bar{\partial} \log |g|^2.
      \end{eqnarray*}
  This inequality combined with (\ref{eq:Berndtsson}) gives
  \begin{eqnarray}\label{eq:Key}
&& \int_{\Omega_\varepsilon} |S \omega|^2 (-\rho)^{r} e^{-\varphi} + \frac1{1-r}
 \int_{\Omega_\varepsilon} |T^\ast_\varphi \omega|^2 (-\rho)^{r} e^{-\varphi}\\
 & \ge & \int_{\Omega_\varepsilon} \sum_{j,k} \Theta_{j\bar{k}} \omega_j\bar{\omega}_k (-\rho)^{r} e^{-\varphi} 
   + \int_{\Omega_\varepsilon} \sum_{j,k}\left|\frac{\partial \omega_j}{\partial \bar{z}_k}\right|^2 (-\rho)^{r} e^{-\varphi}\nonumber
  \end{eqnarray}   
  where $\Theta_{j\bar{k}}=\delta_{jk}+\partial^2 \log |g|^2/\partial z_j\partial \bar{z}_k$.
  
  Now we  apply the trick of Wolff.  Set
    \begin{equation}\label{eq:h}
    \left\{
    \begin{array}{ll}
     h_1 = f\bar{g}_1/|g|^2-u g_2 &\\
     h_2 =  f\bar{g}_2/|g|^2+u g_1 &
    \end{array}
    \right.
    \end{equation}
    and $h=(h_1,h_2)$. Clearly we have $g\cdot h=f$. In order to make $h$ holomorphic, the function $u$ has to satisfy
    $$
    \left\{
    \begin{array}{ll}
     g_2 \bar{\partial} u = f\bar{\partial}(\bar{g}_1/|g|^2) &\\
     g_1 \bar{\partial} u =- f\bar{\partial}(\bar{g}_2/|g|^2),&
    \end{array}
    \right.
    $$
    which actually reduces to a single equation
    $$
    \bar{\partial} u=\overline{(g_2\partial g_1-g_1\partial g_2)}\,f/|g|^4=:v.
    $$
    Note that
    $$
    \partial\bar{\partial} \log |g|^2 =|g|^{-4}(g_2\partial g_1-g_1\partial g_2)\wedge \overline{(g_2\partial g_1-g_1\partial g_2)},
    $$
    which implies
    $$
    |v|^2_{\Theta}\lesssim |f|^2
    $$
    where $
    \Theta=i\sum_{j,k}\Theta_{j\bar{k}} dz_j\wedge d\bar{z}_k$ and $|\cdot|_\Theta$ stands for the point-wise norm with respect to $\Theta$.
    
Suppose $d$ is $C^2$ outside a compact set $K\subset \Omega$.  Take a smooth cut-off function $0\le \chi\le 1$ such that $\chi=1$ near $\partial \Omega$ and $\chi=0$ in a neighborhood of $K$.     Similar as \cite{Andersson}, we consider the operator
     $$
     L=\chi |\partial d|^{-2}{\sum}_j \partial (-d)/\partial \bar{z}_j\cdot \partial/\partial z_j.
     $$
Take a $C^2$ defining function $\hat{\rho}\geq-d$ such that $\hat{\rho}=-d$ near $\partial\Omega$. We shall apply Lemma \ref{lm:Green} with $\rho_\varepsilon:=\hat{\rho}+\varepsilon$, where $0<\varepsilon\ll1$.  We have
    \begin{equation}\label{eq:2generator_1}
    \int_{\Omega_\varepsilon} v\cdot \bar{\omega} e^{-\varphi} = -\int_{\Omega_\varepsilon}
    L(-\rho_\varepsilon) v\cdot \bar{\omega} e^{-\varphi} +  \int_{\Omega_\varepsilon} (1-L(\rho_\varepsilon)) v\cdot \bar{\omega} e^{-\varphi}. 
    \end{equation}
   Suppose $\omega\in {\rm Ker\,}S\cap D_{\mathcal N}(\Omega_\varepsilon)$.    Since $L(\rho_\varepsilon)=L(\hat{\rho})=1$ near $\partial\Omega$, it follows that 
  \begin{eqnarray}
 \left| \int_{\Omega_\varepsilon} (1-L(\rho_\varepsilon)) v\cdot \bar{\omega} e^{-\varphi}\right|^2 &\lesssim&  \left|\int_{\{L(\rho_\varepsilon)\neq 1\}} (-\rho_\varepsilon) v\cdot \bar{\omega}  e^{-\varphi}\right|^2\nonumber \\
 & \lesssim & \int_{\{L(\rho_\varepsilon)\neq1\}} (-\rho_\varepsilon)^{2-r} |v|^2_\Theta e^{-\varphi}\cdot \int_{\{L(\rho_\varepsilon)\neq 1\}} (-\rho_\varepsilon)^r \sum_{j,k} \Theta_{j\bar{k}} \omega_j\bar{\omega}_k e^{-\varphi}\nonumber\\
 & \lesssim &  \int_{\Omega_\varepsilon} (-\rho_\varepsilon) |f|^2 \cdot \int_{\Omega_\varepsilon} (-\rho)^r \sum_{j,k} \Theta_{j\bar{k}} \omega_j\bar{\omega}_k   e^{-\varphi}\nonumber\\
 & \lesssim & \int_{\partial \Omega_\varepsilon} |f|^2 d\sigma_\varepsilon\cdot  \frac1{1-r}
 \int_{\Omega_\varepsilon} |T^\ast_\varphi \omega|^2 (-\rho)^{r} e^{-\varphi}\nonumber\\
 & \lesssim & \int_{\partial \Omega} |f|^2 d\sigma \cdot  \frac1{1-r}
 \int_{\Omega_\varepsilon} |T^\ast_\varphi \omega|^2 (-\rho)^{r} e^{-\varphi}\label{eq:2generator_2}
  \end{eqnarray}  
 where the third inequality follows from the fact that $-\rho_\varepsilon\le d\asymp -\rho$ on $\Omega_\varepsilon$ and the fourth inequality follows from  Lemma \ref{lm:Green} and \eqref{eq:Key}.  Here and in what follows the implicit constants depend only on $\Omega$ and $\|\varphi\|_{L^\infty(\Omega)}$.
    Integral by parts gives
     \begin{eqnarray*}
      -\int_{\Omega_\varepsilon} L(-\rho_\varepsilon)v\cdot \bar{\omega} e^{-\varphi}
     & = & -\int_{\Omega_\varepsilon}  \sum_j \chi |\partial d|^{-2}\frac{\partial (-d)}{\partial \bar{z}_j} \frac{\partial (-\rho_\varepsilon)}{\partial z_j}\, v\cdot 
    \bar{\omega} e^{-\varphi}\\
    & = & \int_{\Omega_\varepsilon} (-\rho_\varepsilon) \sum_j \frac{\partial}{\partial z_j}\left(\chi |\partial d|^{-2} \frac{\partial (-d)}{\partial \bar{z}_j}\right)
              v\cdot  \bar{\omega} e^{-\varphi}\\
        && +   \int_{\Omega_\varepsilon} (-\rho_\varepsilon) \sum_{j,k} \chi  |\partial d|^{-2} \frac{\partial (-d)}{\partial \bar{z}_j}
              \frac{\partial v_k}{\partial z_j}\, \bar{\omega}_k e^{-\varphi}\\
              && +    \int_{\Omega_\varepsilon} (-\rho_\varepsilon) \sum_{j,k} \chi  |\partial d|^{-2} \frac{\partial (-d)}{\partial \bar{z}_j}
              v_k \overline{\frac{\partial \omega_k}{\partial \bar{z}_j}} e^{-\varphi}\\
              && - \int_{\Omega_\varepsilon} (-\rho_\varepsilon) \sum_j \chi  |\partial d|^{-2} \frac{\partial (-d)}{\partial \bar{z}_j}\frac{\partial \varphi}{\partial z_j}\,
              v\cdot  \bar{\omega} e^{-\varphi}\\
               &=:& I_1+I_2+I_3+I_4.
           \end{eqnarray*}
     Similar as \eqref{eq:2generator_2},  we may verify
            \begin{eqnarray*}
            |I_1|^2
 & \lesssim & \left|\int_{\Omega_\varepsilon}(-\rho_\varepsilon)v\cdot\overline{\omega}e^{-\varphi}\right|^2  \lesssim \|f\|_{H^2(\Omega)}^2 \cdot  \frac1{1-r}
 \int_{\Omega_\varepsilon} |T^\ast_\varphi \omega|^2 (-\rho)^{r} e^{-\varphi}.               
   \end{eqnarray*}
 Set $\overline{g_2\partial g_1-g_1\partial g_2}=:v'=\sum_k v'_k d\bar{z}_k$. Since $\partial v_k/\partial z_j=v_k'\partial (f|g|^{-4})/\partial z_j$, we have
   \begin{eqnarray*}
   |I_2|^2 & \le & \int_{\Omega_\varepsilon} (-\rho_\varepsilon)^{2-r} \left|\sum_{j} \chi |\partial d|^{-2} \frac{\partial (-d)}{\partial \bar{z}_j} \frac{\partial (f|g|^{-4})}{\partial z_j}\right|^2 |v'|^2_\Theta e^{-\varphi}\cdot \int_{\Omega_\varepsilon} (-\rho_\varepsilon)^r \sum_{j,k} \Theta_{j\bar{k}} \omega_j\bar{\omega}_k e^{-\varphi} \\
   & \lesssim & \int_{\Omega_\varepsilon} (-\rho_\varepsilon) \left(|\partial f|^2+|f|^2 |\partial g|^2\right)\cdot \int_{\Omega_\varepsilon}(-\rho)^r \sum_{j,k} \Theta_{j\bar{k}} \omega_j\bar{\omega}_k e^{-\varphi}  \\
    & \lesssim & \int_{\Omega_\varepsilon} (-\rho_\varepsilon) \left(|\partial f|^2+|f|^2 \Delta| g|^2\right)\cdot \frac1{1-r}
 \int_{\Omega_\varepsilon} |T^\ast_\varphi \omega|^2 (-\rho)^{r} e^{-\varphi}\\
       & \lesssim &  \int_{\partial \Omega_\varepsilon} |f|^2 d\sigma_\varepsilon\cdot  \frac1{1-r}
 \int_{\Omega_\varepsilon} |T^\ast_\varphi \omega|^2 (-\rho)^{r} e^{-\varphi}\ \ \ \ \ \ (\text{by}\ \eqref{eq:G_1}\ \text{and}\ \eqref{eq:G_3})\\
 & \lesssim & \|f\|_{H^2(\Omega)}^2 \cdot  \frac1{1-r}
 \int_{\Omega_\varepsilon} |T^\ast_\varphi \omega|^2 (-\rho)^{r} e^{-\varphi} .         
          \end{eqnarray*}
Analogously, we have    
\begin{eqnarray*}
|I_3|^2 & \le &  \int_{\Omega_\varepsilon} (-\rho_\varepsilon)^{2-r} \sum_{j,k} \chi^2 |\partial d|^{-4} \left|\frac{\partial (-d)}{\partial \bar{z}_j}\right|^2  | v_k|^2 e^{-\varphi}\cdot \int_{\Omega_\varepsilon} (-\rho_\varepsilon)^r \sum_{j,k} \left|\frac{\partial \omega_k}{\partial \bar{z}_j}\right|^2e^{-\varphi}\\
& \lesssim &  \int_{\Omega_\varepsilon} (-\rho_\varepsilon) \sum_k |v_k|^2 \cdot \int_{\Omega_\varepsilon} (-\rho)^r \sum_{j,k} \left|\frac{\partial \omega_k}{\partial \bar{z}_j}\right|^2 e^{-\varphi}\\
& \lesssim & \int_{\Omega_\varepsilon} (-\rho_\varepsilon) |f|^2 |\partial g|^2  \cdot  \frac1{1-r}
 \int_{\Omega_\varepsilon} |T^\ast_\varphi \omega|^2 (-\rho)^{r} e^{-\varphi}\\
 & \lesssim & \|f\|_{H^2(\Omega)}^2 \cdot  \frac1{1-r}
 \int_{\Omega_\varepsilon} |T^\ast_\varphi \omega|^2 (-\rho)^{r} e^{-\varphi}\ \ \ (\text{by}\ \eqref{eq:G_1}),
               \end{eqnarray*} 
               and
\begin{eqnarray*}
            |I_4|^2 & \lesssim & \int_{\Omega_\varepsilon} (-\rho_\varepsilon)^{2-r} |\partial \varphi|^2 |v|^2_\Theta e^{-\varphi}\cdot \int_{\Omega_\varepsilon} (-\rho_\varepsilon)^r \sum_{j,k} \Theta_{j\bar{k}} \omega_j\bar{\omega}_k e^{-\varphi}\\
            &\lesssim & \int_{\Omega_\varepsilon} (-\rho_\varepsilon) |\partial \varphi|^2 |f|^2 \cdot \int_{\Omega_\varepsilon} (-\rho)^r \sum_{j,k} \Theta_{j\bar{k}} \omega_j\bar{\omega}_k e^{-\varphi}\\
             & \lesssim & \|f\|_{H^2(\Omega)}^2\cdot  \frac1{1-r}
 \int_{\Omega_\varepsilon} |T^\ast_\varphi \omega|^2 (-\rho)^{r} e^{-\varphi}\ \ \ (\text{by}\ \eqref{eq:G_2}).
                         \end{eqnarray*}
               It follows that 
               $$
               \left|\int_{\Omega_\varepsilon} v\cdot \bar{\omega} e^{-\varphi}\right|^2 \lesssim  \|f\|_{H^2(\Omega)}^2 \cdot \frac1{1-r}
 \int_{\Omega_\varepsilon} |T^\ast_\varphi \omega|^2 (-\rho)^{r} e^{-\varphi}
                $$    
                for all $\omega\in {\rm Ker\,}S\cap D_{\mathcal N}(\Omega_\varepsilon)$, hence for all  $\omega\in {\rm Ker\,}S\cap {\rm Dom}(T^\ast_\varphi)$.  For each $\omega\in {\rm Dom}(T^\ast_\varphi)$, we write $\omega=\omega_1+\omega_2\in {\rm Ker\,}S\oplus ({\rm Ker}\,S)^\bot$. Since the range ${\rm R}(T)$ of $T$ is contained in ${\rm Ker\,}S$, so $T^\ast_\varphi \omega_2=0$. On the other hand, we have $\int_{\Omega_\varepsilon} v\cdot \bar{\omega}_2 e^{-\varphi}=0$. Thus the previous inequality remains valid for all $\omega\in {\rm Dom}(T^\ast_\varphi)$. It follows that the linear functional
                $$
               F: (-\rho)^{r/2} T^\ast_\varphi \omega \rightarrow \int_{\Omega_\varepsilon} \omega\cdot \bar{v} e^{-\varphi}
                                $$
            satisfies $\|F\|\lesssim (1-r)^{-1/2} \|f\|_{H^2(\Omega)}$. The Hahn-Banach theorem combined with the Riesz theorem gives a function  $\tilde{u}_\varepsilon$ on $\Omega_\varepsilon$ which satisfies  $\|\tilde{u}_\varepsilon\|_\varphi\lesssim (1-r)^{-1/2} \|f\|_{H^2(\Omega)}$  and 
            $$
            \int_{\Omega_\varepsilon} v\cdot \bar{\omega} e^{-\varphi} = \int_{\Omega_\varepsilon} \tilde{u}_\varepsilon \cdot \overline{(-\rho)^{r/2} T^\ast_\varphi \omega}\, e^{-\varphi},\ \ \ \forall\,   \omega\in {\rm Dom}(T^\ast_\varphi).
                     $$ 
                     If we set $u_\varepsilon=(-\rho)^{r/2}\tilde{u}_\varepsilon$, then we obtain   $T u_\varepsilon=v$ with 
                     $$
                    (1-r) \int_{\Omega_\varepsilon} |u_\varepsilon|^2 (-\rho)^{-r} e^{-\varphi} \lesssim \|f\|^2_{H^2(\Omega)}.
                     $$
                     We may take a weak limit $u$ of $\{u_\varepsilon\}$ satisfying $\bar{\partial}u=v$ (in the sense of distributions) on $\Omega$ and
                     $$
          (1-r) \int_{\Omega} |u|^2 (-\rho)^{-r}\lesssim  (1-r) \int_{\Omega} |u|^2 (-\rho)^{-r} e^{-\varphi} \lesssim \|f\|^2_{H^2(\Omega)}.
                                          $$
             If $h=(h_1,h_2)$ is given by \eqref{eq:h},  then $h$ is holomorphic and satisfies $h\cdot g=f$,  
             \begin{eqnarray*}
        (1-r) \int_{\Omega} |h|^2 (-\rho)^{-r}  &\lesssim& (1-r) \int_{\Omega} |f|^2 (-\rho)^{-r}  +(1-r) \int_{\Omega} |u|^2 (-\rho)^{-r} \\
              & \lesssim & (1-r) \int_{\Omega} |f|^2 (-\rho)^{-r} +\|f\|^2_{H^2(\Omega)}.
                           \end{eqnarray*}   
                           Letting $r\rightarrow 1-$, we obtain $\|h\|_{H^2(\Omega)}\lesssim \|f\|_{H^2(\Omega)}$.       
            \section{The general case}               
In this section we shall prove Theorem \ref{th:Main}. Let $U$ be a bounded pseudoconvex domain with $C^2-$boundary. For a  weight function $\varphi\in C^2(\overline{U})$,  we define $L^2_{(0,q)}(U,\varphi)^{\oplus m}$ to be the space of vector-valued $(0,q)-$forms $\omega=(\omega_1,\cdots,\omega_m)$, where
$$
\omega_k= {\sum_{|J|=q}}' \omega_{kJ}  d\bar{z}_J\ \ \ \text{with}\ \ \ \int_U |\omega_{kJ}|^2 e^{-\varphi}<\infty,\ \ \ \forall\,k,J.
$$
Given $\omega,\zeta\in L^2_{(0,q)}(U,\varphi)^{\oplus m}$,   define
\begin{eqnarray*}
\langle\omega,\zeta\rangle (z)&:=&\sum^m_{k=1}{\sum_{|J|=q}}' \omega_{kJ}(z)\,\overline{\zeta_{kJ}(z)},\ \ \ z\in{U};\\
(\omega,\zeta)_\varphi&:=&\int_U \langle\omega,\zeta\rangle e^{-\varphi};\\
\|\omega\|^2_\varphi&:=&(\omega,\omega)_\varphi.
\end{eqnarray*}
In what follows,  we shall use the same symbols as above, for the sake of simplicity.   The $\bar{\partial}-$operator 
induces two closed and densely defined  operators
$$
T: L^2 (U,\varphi)^{\oplus m}\rightarrow L^2_{(0,1)}(U,\varphi)^{\oplus m},\ \ \  S: L^2_{(0,1)}(U,\varphi)^{\oplus m}\rightarrow L^2_{(0,2)}(U,\varphi)^{\oplus m}.
$$  
Set
$$
D_{\mathcal N}(U)^{\oplus m}=\left\{\omega\in C^1_{(0,1)}(\overline{U})^{\oplus m}: {\sum}_{\nu=1}^n  \omega_{k\nu} {\partial \rho}/{\partial z_\nu}=0\ \text{on\ }\partial{U},  1\le k\le m\right\}.
$$
 Let $T^\ast_\varphi$  denote  the adjoint of $T$ with respect to the inner product $(\cdot,\cdot)_\varphi$.   Since 
$$
S \omega=(S \omega_1,\cdots, S \omega_m)\ \ \ \text{and}\ \ \ T^\ast_\varphi \omega=(T^\ast_\varphi \omega_1,\cdots, T^\ast_\varphi \omega_m),
$$
it follows that  $D_{\mathcal N}(U)^{\oplus m}$  lies dense in ${\rm Dom}(T^\ast_\varphi)\cap {\rm Dom}(S)$ with respect to the graph norm
$$
\omega\rightarrow \|\omega\|_{\varphi}+\|T^\ast_\varphi \omega\|_{\varphi}+\|S \omega\|_{\varphi}
$$
and
\begin{eqnarray}\label{eq:m-MKH}
 \|S \omega\|_\varphi^2 + \|T^\ast_\varphi \omega\|_\varphi^2 
& \geq &  \sum_{k=1}^m\sum_{\mu,\nu=1}^n\int_{U} \frac{\partial^2 \varphi }{\partial z_\mu\partial \bar{z}_\nu} \omega_{k\mu}\bar{\omega}_{k\nu} e^{-\varphi}\\
&& + \sum_{k=1}^m\sum_{\mu,\nu=1}^n\int_{U}\left|\frac{\partial\omega_{k\mu}}{\partial\bar z_\nu}\right|^2 e^{-\varphi},\nonumber
\end{eqnarray}
 for all $\omega\in D_{\mathcal N}(U)^{\oplus m}$.

Now let $\Omega$ be as Theorem \ref{th:Main} and suppose $\varphi\in C^2(\Omega)$.  Again we set 
$
\Omega_\varepsilon=\{d >\varepsilon\},
$
where $0<\varepsilon\ll1$. 
For $g=(g_1,g_2,\cdots,g_m)$, we write 
$$
|g|^2=|g_1|^2+\cdots+|g_m|^2,\ \ \ |\partial g|^2=|\partial g_1|^2+\cdots+|\partial g_m|^2.
$$
Consider the following vector-valued $\bar{\partial}-$equation on $\Omega_\varepsilon$
\begin{eqnarray}\label{eq:dbar}
T u=v:=f\bar{\partial}(\bar g/|g|^2),\ \ \ g\cdot u=0.
\end{eqnarray}
Suppose \eqref{eq:dbar} admits a solution $u_\varepsilon$ with the following estimate 
\begin{equation}\label{eq:general_1}
          (1-r) \int_{\Omega_\varepsilon} |u_\varepsilon|^2 (-\rho)^{-r} \leq C(\delta)^2  \|f\|^2_{H^2(\Omega)},\ \ \ \forall\,0<r<1,
  \end{equation}
  where  $C(\delta)=C\delta^{-2}|\log \delta|$, $0<\delta\ll1$, and $C$ is a generic constant independent of $m,\varepsilon,r$. Let $u$ be a weak limit of $u_\varepsilon$.  It follows that   
           $h:=f\bar{g}/|g|^2-u\in \mathcal O(\Omega)^{\oplus m}$ satisfies $h\cdot g=f$ and 
\begin{eqnarray*}
          \|h\|^2_{H^2(\Omega)} &=& \lim_{r\rightarrow 1-} (1-r) \int_{\Omega} |h|^2 (-\rho)^{-r} \nonumber\\
         & \le & C \delta^{-2} \|f\|^2_{H^2(\Omega)}+ C\delta^{-4}|\log \delta|^2 \|f\|^2_{H^2(\Omega)}\\
         &  \le & C(\delta)^2 \|f\|^2_{H^2(\Omega)} .
\end{eqnarray*}

          It remains to find a solution of \eqref{eq:dbar} which verifies \eqref{eq:general_1}.  
 Set
$$
\mathcal{H}_q(\Omega_\varepsilon)=\left\{\omega\in L^2_{(0,q)}(\Omega_\varepsilon,\varphi)^{\oplus m}:g\cdot \omega =0\right\},\ \ \ q=0,1,2,
$$
and
$$
\mathcal{D}_1 (\Omega_\varepsilon)=\left\{\omega\in D_\mathcal N(\Omega_\varepsilon)^{\oplus m}: g\cdot \omega =0\right\}.
$$
It is not difficult to verify that the orthogonal complement of the subspace $\mathcal{H}_0(\Omega_\varepsilon)$ in $L^2(\Omega_\varepsilon,\varphi)^{\oplus m}$ is given by 
$$
\mathcal{H}_0(\Omega_\varepsilon)^\bot=\{c \bar g: c\in L^2(\Omega_\varepsilon,\varphi)\}.
$$
Clearly,  the set
$$
\mathcal{D}_0^\bot(\Omega_\varepsilon):=\left\{c \bar g: c\in C^\infty_0(\Omega_\varepsilon) \right\}
$$
lies dense in $\mathcal{H}_0(\Omega_\varepsilon)^\bot$.

Since $g$ is holomorphic,  it follows that the $\bar{\partial}-$operator induces two closed and densely defined  operators
$$
T_{\mathcal H}: \mathcal H_0(\Omega_\varepsilon)\rightarrow \mathcal H_1(\Omega_\varepsilon),\ \ \ S_{\mathcal H}: \mathcal H_1(\Omega_\varepsilon)\rightarrow \mathcal H_2(\Omega_\varepsilon).
$$
	Let $T^\ast_{\mathcal{H},\varphi}$ be the adjoint of $T_\mathcal H$.   
	The relationship between $T^\ast_\varphi$ and $T^\ast_{\mathcal{H},\varphi}$ is given as follows.
	
\begin{lemma}[cf. \cite{Ohsawa},  p.  89]\footnote{The original formulation in \cite{Ohsawa} seems incorrect.}\label{lm:Ohsawa}
For any $\omega\in\mathcal{D}_1(\Omega_\varepsilon)$, we have
$$
T^\ast_{\mathcal{H},\varphi}\omega=T^\ast_{\varphi}\omega-\left(\sum_{k=1}^m\sum_{\nu=1}^n \frac{\partial}{\partial z_\nu}(g_k/|g|^2) \omega_{k\nu} \right) \bar g=:T^\ast_{\varphi}\omega-\Phi_g \omega.
$$
\end{lemma}

\begin{proof}
We include a proof here for the sake of completeness.  
Since $L^2(\Omega_\varepsilon,\varphi)$ is a separable Hilbert space and $C^\infty_0(\Omega_\varepsilon)$ lies dense in $L^2(\Omega_\varepsilon,\varphi)$, we may choose by the Gram-Schmidt method a complete orthonormal basis $\{{\bf e}_j\}\subset \mathcal{D}_0^\bot(\Omega_\varepsilon)$ of $\mathcal{H}_0(\Omega_\varepsilon)^\bot$.  Let $P\omega$ be the orthogonal projection of $T^\ast_\varphi \omega$ to $\mathcal{H}_0(\Omega_\varepsilon)$,  i.e.,
$$
P\omega =T^\ast_\varphi \omega - \sum_j (T^\ast_\varphi \omega,{\bf e}_j)_\varphi {\bf e}_j.
$$
Put ${\bf e}_j=\chi_j \bar{g}/|g|$.  Clearly $\chi_j\in C^\infty_0(\Omega_\varepsilon)$ and $\{\chi_j\}$ forms a complete orthonormal basis of $L^2 (\Omega_\varepsilon,\varphi)$.  Since $g\cdot \omega=0$  and
$$
\bar{\partial} {\bf e}_j=\bar{\partial}(|g|\chi_j) \bar{g}/|g|^2+|g|\chi_j \bar{\partial}(\bar{g}/|g|^2),
$$
it follows that
\begin{eqnarray*}
(T^\ast_\varphi \omega,{\bf e}_j)_\varphi & = & (\omega, \bar{\partial}{\bf e}_j)_\varphi=(\omega,|g|\chi_j \bar{\partial}(\bar{g}/|g|^2))_\varphi\\
& = & \sum_k \sum_\nu \int_{\Omega_\varepsilon} |g| \frac{\partial}{\partial z_\nu}(g_k/|g|^2) \omega_{k\nu} \bar \chi_j e^{-\varphi}.
\end{eqnarray*}
Thus
\begin{eqnarray*}
P \omega & = & T^\ast_\varphi \omega- \sum_j  \int_{\Omega_\varepsilon} G \bar \chi_j e^{-\varphi} \,{\bf e}_j,
\end{eqnarray*}
where
\[
G:=|g|\sum_k\sum_\nu\frac{\partial}{\partial z_\nu}(g_k/|g|^2) \omega_{k\nu}.
\]
Since $\{\chi_j\}$ is a complete orthonormal basis of $L^2 (\Omega_\varepsilon,\varphi)$, we have
\begin{eqnarray*}
P\omega
& = & T^\ast_\varphi\omega-\sum_j(G,\chi_j)_{L^2(\Omega_\varepsilon,\varphi)}{\bf e}_j\\
& = & T^\ast_\varphi\omega-\left\{{\sum}_j(G,\chi_j)_{L^2(\Omega_\varepsilon,\varphi)}\chi_j\right\}\frac{\bar{g}}{|g|}\\
& = & T^\ast_\varphi\omega-G\frac{\bar{g}}{|g|}\\
& = & T^\ast_\varphi \omega-\Phi_g\omega,
\end{eqnarray*}
where $(\cdot,\cdot)_{L^2(\Omega_\varepsilon,\varphi)}$ is the inner product of the Hilbert space $L^2(\Omega_\varepsilon,\varphi)$.

For any $\eta \in \mathcal H_0(\Omega_\varepsilon)\cap C_0^\infty(\Omega_\varepsilon)^{\oplus m}$, we have
$$
(T^\ast_{\mathcal H,\varphi} \omega,\eta)_\varphi=(\omega,\bar{\partial} \eta)_\varphi=(T^\ast_\varphi \omega,\eta)_\varphi=(P\omega,\eta)_\varphi.
$$
Since $\mathcal H_0(\Omega_\varepsilon)\cap C_0^\infty(\Omega_\varepsilon)^{\oplus m}$ lies dense in $H_0(\Omega_\varepsilon)$,  it follows that $T^\ast_{\mathcal H,\varphi} \omega=P\omega$.
\end{proof}

\begin{lemma}\label{lm:density}
	$\mathcal D_1(\Omega_\varepsilon)$  lies dense in $\mathrm{Dom}(T^\ast_{\mathcal{H},\varphi})\cap \mathrm{Dom}(S_{\mathcal H})$ with respect to the graph norm
	$$
	\omega\mapsto \|\omega\|_\varphi + \| T^\ast_{\mathcal{H},\varphi} \omega\|_\varphi + \| S_{\mathcal H}\omega\|_\varphi.
	$$
	\end{lemma}
	
	\begin{proof}
	Given $\omega\in \mathrm{Dom}(T^\ast_{\mathcal{H},\varphi})\cap \mathrm{Dom}(S_{\mathcal H})$,  take a sequence $\{\omega^j\}\subset D_\mathcal N(\Omega_\varepsilon)^{\oplus m}$ such that 
	$$
	\|\omega^j-\omega\|_\varphi \rightarrow 0,\ \|S\omega^j-S\omega\|_\varphi \rightarrow 0,\ \|T^\ast_\varphi \omega^j - T^\ast_\varphi \omega\|_\varphi\rightarrow 0.
	$$
	Set $\tilde{\omega}^j:=\omega^j-\frac{g\cdot \omega^j}{|g|^2}\bar{g}$.  Since $g\cdot \tilde{\omega}^j=g\cdot \omega^j-\frac{g\cdot \omega^j}{|g|^2}|g|^2=0$,  we see that $\tilde{\omega}^j\in \mathcal H_1(\Omega_\varepsilon)$.  Moreover,  since $g\cdot \omega=0$,  we have
	$$
	\tilde{\omega}^j - \omega = \omega^j-\omega - \frac{g\cdot \omega^j}{|g|^2}\bar{g} =  \omega^j-\omega - \frac{g\cdot (\omega^j-\omega)}{|g|^2}\bar{g},
	$$ 
	so that $\|\tilde{\omega}^j-\omega\|_\varphi \rightarrow 0$,  for both $g$ and $|g|^{-1}$ are bounded. 
	
	 Analogously,  since 
	\begin{eqnarray*}
	S \tilde{\omega}^j  & = & S \omega^j + g\cdot\omega^j\wedge S(\bar{g}/|g|^2) -S(g\cdot\omega^j)\cdot \bar{g}/|g|^2\\
	& = & S \omega^j + g\cdot\omega^j\wedge S(\bar{g}/|g|^2) - g\cdot S(\omega^j)\cdot \bar{g}/|g|^2, 
	\end{eqnarray*}
	and $ g\cdot S\omega = S(g\cdot \omega)  =0$,  we may write 
	$$
	S\tilde\omega^j-S\omega = S(\omega^j-\omega) + g\cdot(\omega^j-\omega) S(\bar{g}/|g|^2) - g\cdot(S\omega^j-S\omega) \bar{g}/|g|^2,
	$$
	so that $\|S\tilde{\omega}^j-S\omega\|_\varphi\rightarrow 0$,  and $\|S_\mathcal H \tilde{\omega}^j-S_\mathcal H\omega\|_\varphi\rightarrow 0$.
	
	It remains to verify $\|T^\ast_{\mathcal H,\varphi}\tilde{\omega}^j- T^\ast_{\mathcal H,\varphi}\tilde{\omega}\|_\varphi \rightarrow 0$.  In view of Lemma \ref{lm:Ohsawa} and the fact $\|\tilde{\omega}^j-\omega\|_\varphi \rightarrow 0$,  it suffices to show
	$$
	\|T^\ast_\varphi \tilde{\omega}^j-T^\ast_\varphi \omega \|_\varphi \rightarrow 0.
	$$
	This in turn follows from the following facts:
	$$
	T^\ast_\varphi \tilde{\omega}^j = T^\ast_\varphi \omega^j - T^\ast_\varphi (g\cdot \omega^j \bar{g}/|g|^2)
	$$
	\begin{eqnarray*}
	T^\ast_\varphi (g\cdot \omega^j \bar{g}_k/|g|^2) & = & T^\ast_\varphi (g\cdot (\omega^j-\omega) \bar{g}_k/|g|^2) \\
	& = & \bar{g}_k T^\ast_\varphi (g\cdot (\omega^j-\omega) /|g|^2) + \sum_{\nu=1}^n \frac{\partial \bar{g}_k}{\partial z_\nu} \frac{(g\cdot(\omega^j-\omega))_\nu}{|g|^2}\\
	& = & \bar{g}_k T^\ast_\varphi (g\cdot (\omega^j-\omega) /|g|^2),\ \forall\,1\le k\le m
	\end{eqnarray*}
and
\begin{eqnarray*}
T^\ast_\varphi (g\cdot (\omega^j-\omega)/ |g|^2) & = & \sum_{l=1}^m T^\ast_\varphi(g_l(\omega^j-\omega)_l/|g|^2)\\
& = & \sum_{l=1}^m \frac{g_l}{|g|^2} T^\ast_\varphi (\omega^j-\omega)_l +\sum_{l=1}^m \sum_{\nu=1}^n \frac{\partial}{\partial z_\nu} \left(\frac{g_l}{|g|^2}\right) (\omega^j-\omega)_{l,\nu}\\
& = & \frac1{|g|^2} g\cdot T^\ast_\varphi(\omega^j-\omega) + \sum_{l=1}^m \sum_{\nu=1}^n \frac{\partial}{\partial z_\nu} \left(\frac{g_l}{|g|^2}\right) (\omega^j-\omega)_{l,\nu}
\end{eqnarray*}	
(note that $\partial (g_l/|g|^2)/\partial z_\nu$ is bounded on $\Omega_\varepsilon$).  Here
$$
\omega^j-\omega =((\omega^j-\omega)_1,\cdots,(\omega^j-\omega)_m),\ \ \  (\omega^j-\omega)_l=\sum_{\nu=1}^n (\omega^j-\omega)_{l,\nu}.
$$
	\end{proof}

Since the vector $T^\ast_{\mathcal{H},\varphi}\omega|_z$ is orthogonal to $\Phi_g \omega|_z$ for any $z\in \Omega_\varepsilon$,  we infer from Lemma \ref{lm:Ohsawa} that
$$
|T^\ast_{\varphi}\omega|^2=|\Phi_g \omega|^2+|T^\ast_{\mathcal{H},\varphi}\omega|^2
$$ 
holds point-wisely in $\Omega_\varepsilon$.  Again, let $\phi=-r\log(-\rho)$ where $0<r<1$.  Since
$$
T^\ast_{\mathcal{H},\varphi+\phi}\omega+\Phi_g \omega=T^\ast_{\varphi+\phi}\omega=T^\ast_{\varphi}\omega+\sum_{k=1}^m \sum_{\nu=1}^n \frac{\partial \phi}{\partial z_\nu} \omega_{k\nu}=T^\ast_{\mathcal{H},\varphi}\omega+\sum_{k=1}^m \sum_{\nu=1}^n \frac{\partial \phi}{\partial z_\nu} \omega_{k\nu} +\Phi_g \omega,
$$
i.e.,
$$
T^\ast_{\mathcal{H},\varphi+\phi}\omega=T^\ast_{\mathcal{H},\varphi}\omega+\sum_{k=1}^m \sum_{\nu=1}^n \frac{\partial \phi}{\partial z_\nu} \omega_{k\nu},
$$
it follows from the Cauchy-Schwarz inequality that
$$
\left|T^\ast_{\varphi+\phi}\omega\right|^2=\left|\Phi_g \omega\right|^2+\left|T^\ast_{\mathcal{H},\varphi+\phi}\omega\right|^2\leq\left|\Phi_g \omega\right|^2+(1-r)^{-1}\left|T^\ast_{\mathcal{H},\varphi}\omega\right|^2+r^{-1}\sum_{k=1}^m\left|  \sum_{\nu=1}^n \frac{\partial \phi}{\partial z_\nu} \omega_{k\nu} \right|^2.
$$ 
Applying (\ref{eq:m-MKH}) with $U$ and $\varphi$ replaced by $\Omega_\varepsilon$ and $\varphi+\phi$ respectively,  we obtain
\begin{eqnarray}\label{eq:n-Berndtsson}
&&\int_{\Omega_\varepsilon}|\Phi_g \omega|^2(-\rho)^r e^{-\varphi}+(1-r)^{-1}\int_{\Omega_\varepsilon}|T^\ast_{\mathcal{H},\varphi}\omega|^2(-\rho)^r e^{-\varphi}
+\int_{\Omega_\varepsilon}|S_{\mathcal H}\omega|^2(-\rho)^r e^{-\varphi}\\
&&\ge\int_{\Omega_\varepsilon}  \langle [i\partial\bar{\partial}\varphi+r(-\rho)^{-1}i\partial\bar{\partial}\rho,\Lambda]\omega,\omega\rangle (-\rho)^r e^{-\varphi}\nonumber\\
&& +\sum_{k=1}^m\sum_{\mu,\nu=1}^n \int_{\Omega_\varepsilon}\left|\frac{\partial\omega_{k\mu}}{\partial\bar z_\nu}\right|^2(-\rho)^r e^{-\varphi},\nonumber
\end{eqnarray}
for any $\omega\in \mathcal{D}_1(\Omega_\varepsilon)$.  Here we adopt the nonstandard notion that if $\Theta=i\sum_{\mu,\nu=1}^n \Theta_{\mu\bar\nu} dz_\mu\wedge d\bar z_\nu$ is a continuous $(1,1)-$form,  then
$$
\langle [\Theta,\Lambda] \omega,\omega\rangle:= \sum_{k=1}^m\sum_{\mu,\nu=1}^n\Theta_{\mu\bar\nu} \omega_{k\mu}\bar{\omega}_{k\nu}.  
$$
Traditionally,  the RHS of the previous equation is denoted by
$$
\langle [\Theta,\Lambda] dz_1 \wedge\cdots\wedge dz_n\wedge \omega,dz_1 \wedge\cdots\wedge dz_n\wedge\omega\rangle
$$ 
(see Demailly \cite{DemaillyBook}).  

To proceed the proof,  we need the following fundamental inequality due to Skoda:
 
\begin{lemma}[cf. \cite{Skoda}]\label{le:skoda}
For any matrix $\zeta=(\zeta_{k\mu})_{m\times n}$, we have 
$$
q\sum_k\sum_{\mu,\nu}\frac{\partial^2\log|g|^2}{\partial z_\mu\bar z_\nu}\zeta_{k\mu}\bar\zeta_{k\nu}\ge 
|g|^2\left|\sum_k\sum_\nu\frac{\partial}{\partial\bar z_\nu}(\bar g_k/|g|^2)\bar\zeta_{k\nu}\right|^2,
$$
where $q=\min\{n,m-1\}$.
\end{lemma}

Set 
$\varphi(z):=|z|^2+\lambda(z)+q\log|g(z)|^2+\psi_\delta$, where   $\psi_\delta:=\frac{\log|g(z)|^2}{|\log \delta|}$ satisfies 
\begin{equation}\label{eq:psi_delta}
-2\le \psi_\delta\le 0.
\end{equation}
Based on Lemma \ref{le:skoda},  we obtain by a similar calculation as  the case $m=2$ the following inequality: 
$$
\langle [i\partial\bar{\partial}\varphi+r(-\rho)^{-1}i\partial\bar{\partial}\rho,\Lambda]\omega,\omega\rangle \ge \langle [\Theta,\Lambda]\omega,\omega\rangle + |\Phi_g \omega|^2, 
$$
where
$$ 
\Theta=i\partial\bar{\partial}|z|^2+\frac{i\partial\bar{\partial}\log|g|^2}{|\log\delta|}.  
$$
This together with (\ref{eq:n-Berndtsson}) yield
\begin{eqnarray}\label{eq:KeyIneqInGeneral}
&&(1-r)^{-1}\int_{\Omega_\varepsilon}|T^\ast_{\mathcal{H},\varphi}\omega|^2(-\rho)^r e^{-\varphi}
+\int_{\Omega_\varepsilon}|S_\mathcal H\omega|^2(-\rho)^r e^{-\varphi}\\
&&\ge\int_{\Omega_\varepsilon}\langle[\Theta,\Lambda]\omega,\omega\rangle(-\rho)^r e^{-\varphi}+\sum_{k=1}^m\sum_{\mu,\nu=1}^n \int_{\Omega_\varepsilon}\left|\frac{\partial\omega_{k\mu}}{\partial\bar z_\nu}\right|^2(-\rho)^r e^{-\varphi}.\nonumber
\end{eqnarray}

\begin{lemma}\label{lm:elembound}
For $v=f\bar{\partial} (\bar g/|g|^2)$, we have 
\begin{equation}\label{eq:v-bound}
\langle[\Theta,\Lambda]^{-1}v,v\rangle:= \sup\{|\langle v,\omega\rangle|^2:\langle[\Theta,\Lambda]\omega,\omega\rangle\leq 1\} \leq \frac{q|\log\delta|}{\delta^2}|f|^2.
\end{equation}
\end{lemma}

\begin{proof}
Set $\Gamma:=\frac{ i\partial\bar{\partial}\log|g|^2}{|\log\delta|}$.  Thanks to Lemma \ref{le:skoda}, we have
\begin{eqnarray*}
1 & \geq & \langle[\Theta,\Lambda]\omega,\omega\rangle \ge \langle[\Gamma,\Lambda]\omega,\omega\rangle \\
& = &\frac{1}{|\log\delta|} \sum_k\sum_{\mu,\nu}\frac{\partial^2\log |g|^2}{\partial z_\mu\partial \bar z_\nu}\omega_{k\mu}\bar \omega_{k\nu} \\
& \geq &\frac{|g|^2}{q|\log\delta|} \left|\sum_{k,\nu}\frac{\partial}{\partial\bar z_\nu}(\bar g_k/|g|^2)\bar \omega_{k\nu}\right|^2,
\end{eqnarray*}
so that
\begin{eqnarray*}
|\langle v,\omega\rangle|^2 & = & \left|\sum_{k,\nu}f\frac{\partial}{\partial\bar z_\nu}(\bar g_k/|g|^2)\bar \omega_{k\nu}\right|^2\\
& \leq & \frac{q|f|^2|\log\delta|} {|g|^2}\leq\frac{q|\log\delta|}{\delta^2}|f|^2.
\end{eqnarray*}
\end{proof}

In view of the decomposition \eqref{eq:2generator_1}, we only need to estimate
\begin{eqnarray*}
-\int_{\Omega_\varepsilon}L(-\rho_\varepsilon)\langle v,\omega\rangle e^{-\varphi} & = &  -\int_{\Omega_\varepsilon} \chi |\partial d|^{-2}\sum_{\nu=1}^n\frac{\partial (-d)}{\partial\bar z_\nu}\,\frac{\partial(-\rho_\varepsilon)}{\partial z_\nu}\langle v,\omega\rangle e^{-\varphi}\\
&=&\int_{\Omega_\varepsilon}(-\rho_\varepsilon)\sum_{\nu=1}^n\frac{\partial}{\partial z_\nu}\left( \chi |\partial d|^{-2}\frac{\partial(-d)}{\partial \bar z_\nu}\right)\langle v,\omega\rangle e^{-\varphi}\\
&&+\sum_{k=1}^m\int_{\Omega_\varepsilon}(-\rho_\varepsilon)\sum_{\mu,\nu=1}^n \chi |\partial d|^{-2}\frac{\partial(-d)}{\partial \bar z_\nu}\frac{\partial v_{k\mu}}{\partial z_\nu}\bar \omega_{k\mu}e^{-\varphi}\\
&&+\sum_{k=1}^m\int_{\Omega_\varepsilon}(-\rho_\varepsilon)\sum_{\mu,\nu=1}^n \chi |\partial d|^{-2}\frac{\partial(-d)}{\partial \bar z_\nu}v_{k\mu}\overline{\frac{\partial \omega_{k\mu}}{\partial\bar z_\nu}}e^{-\varphi}\\
&&-\int_{\Omega_\varepsilon}(-\rho_\varepsilon)\sum_{\nu=1}^n \chi |\partial d|^{-2}\frac{\partial(-d)}{\partial \bar z_\nu}{\frac{\partial \varphi}{\partial z_\nu}}\langle v,\omega\rangle e^{-\varphi}\\
&=:&I_1+I_2+I_3+I_4.
\end{eqnarray*}

We also need the following elementary fact:
\begin{lemma}\label{lm:Green2}
\begin{eqnarray}\label{eq:g}
\int_{\Omega_\varepsilon}(-\rho_\varepsilon)|f|^2|\partial g|^2/|g|^4\lesssim \frac{|\log\delta|}{\delta^2} \|f\|^2_{H^2(\Omega)},
\end{eqnarray}
\begin{eqnarray}\label{eq:ag}
\int_{\Omega_\varepsilon}(-\rho_\varepsilon)|f|^2\left|\partial |g|\,\right|^2/|g|^{4}\lesssim  \frac{|\log\delta|}{\delta^2} \|f\|^2_{H^2(\Omega)}.
\end{eqnarray}
 \end{lemma}
 
\begin{proof}
 For $\psi:=\log\frac{|g|^2+\delta^2}{2}$, we have $\log\delta\lesssim\psi<0$ and 
 \begin{eqnarray*}
 i\partial\bar\partial\psi&=&\frac{(|g|^2+\delta^2)i\partial\bar\partial|g|^2-i\partial|g|^2\wedge\bar\partial|g|^2}{(|g|^2+\delta^2)^2}\\
 &\geq&\frac{\delta^2i\partial\bar\partial|g|^2}{(|g|^2+\delta^2)^2}\geq\frac{\delta^2}{4|g|^4}i\partial\bar\partial|g|^2,
 \end{eqnarray*}
 so that
 $$
 \Delta\psi\geq\frac{\delta^2}{4|g|^4}\Delta|g|^2=\frac{\delta^2}{|g|^4}|\partial g|^2.
 $$
 Apply \eqref{eq:G_1} with $\psi=\log\frac{|g|^2+\delta^2}{2}$,  we  get (\ref{eq:g}).  
 
Since $\partial |g|^2 =\sum_{k=1}^m \bar{g}_k \partial g_k$,   the Cauchy-Schwarz inequality yields
$$
\left|\partial |g|^2\right|\le |g||\partial g|,
$$
so that $|\partial |g||\le |\partial g|/2$.  This combined with \eqref{eq:g} gives \eqref{eq:ag}.  
 \end{proof}
 
For all $\omega\in\text{Ker}\,S_{\mathcal{H}}\cap \mathcal{D}_1(\Omega_\varepsilon)$, we have
\begin{eqnarray*}
|I_1|^2&\lesssim&\int_{\Omega_\varepsilon}(-\rho_\varepsilon)^{2-r} \langle[\Theta,\Lambda]^{-1} v,  v\rangle e^{-\varphi}\cdot\int_{\Omega_\varepsilon}(-\rho_\varepsilon)^r\langle[\Theta,\Lambda]\omega,\omega\rangle e^{-\varphi}\\
&\lesssim& \frac{q|\log\delta|}{\delta^{2+2q}} \int_{\Omega_\varepsilon}(-\rho_\varepsilon)|f|^2 \cdot\int_{\Omega_\varepsilon}(-\rho)^r\langle[\Theta,\Lambda]\omega,\omega\rangle e^{-\varphi}\ \ \ (\text{by\ }\eqref{eq:v-bound})\\
&\lesssim& \frac{q|\log\delta|}{\delta^{2+2q}} \int_{\partial\Omega_{\varepsilon}}|f|^2d\sigma_\varepsilon\cdot\frac{1}{1-r}\int_{\Omega_\varepsilon}|T^\ast_{\mathcal{H},\varphi}\omega|^2(-\rho)^r e^{-\varphi}\ \ \ (\text{by\ }\eqref{eq:KeyIneqInGeneral})\\
&\lesssim& \frac{q|\log\delta|}{\delta^{2+2q}}  \|f\|_{H^2(\Omega)}^2\cdot\frac{1}{1-r}\int_{\Omega_\varepsilon}|T^\ast_{\mathcal{H},\varphi}\omega|^2(-\rho)^r e^{-\varphi}.  
\end{eqnarray*}
Define
$$
Q(r,\varepsilon):=\frac{1}{1-r}\int_{\Omega_\varepsilon}|T^\ast_{\mathcal{H},\varphi}\omega|^2(-\rho)^r e^{-\varphi}.
$$
Analogously,  we have
\begin{eqnarray*}
|I_3|^2 &\leq& \int_{\Omega_\varepsilon}(-\rho_\varepsilon)^{2-r}\sum_{\nu=1}^n \chi^2 |\partial{d}|^{-4}\left|\frac{\partial(-d)}{\partial\bar z_\nu}\right|^2|v|^2 e^{-\varphi}\cdot\int_{\Omega_\varepsilon}\sum_{k=1}^m\sum_{\mu,\nu=1}^n\left|\frac{\partial\omega_{k\mu}}{\partial\bar z_\nu}\right|^2(-\rho_\varepsilon)^r e^{-\varphi}\\
&\lesssim& \delta^{-2q} \int_{\Omega_\varepsilon}(-\rho_\varepsilon)|v|^2 \cdot\int_{\Omega_\varepsilon}\sum_{k=1}^m\sum_{\mu,\nu=1}^n\left|\frac{\partial\omega_{k\mu}}{\partial\bar z_\nu}\right|^2(-\rho)^r e^{-\varphi}\\
&\lesssim& \delta^{-2q} \int_{\Omega_\varepsilon}(-\rho_\varepsilon)|f|^2|\partial g|^2 {|g|^{-4}}\cdot Q(r,\varepsilon)\ \ \ (\text{by\ }\eqref{eq:KeyIneqInGeneral})\\
&\lesssim&  \frac{|\log\delta|}{\delta^{2+2q}} \|f\|_{H^2(\Omega)}^2\cdot Q(r,\varepsilon)\ \ \ (\text{by\ }\eqref{eq:g})
\end{eqnarray*}
and 
\begin{eqnarray*}
|I_4|^2 &\lesssim& \int_{\Omega_\varepsilon}(-\rho_\varepsilon)^{2-r} |\partial\varphi|^2 \langle[\Theta,\Lambda]^{-1} v,  v\rangle  e^{-\varphi}\cdot\int_{\Omega_\varepsilon}(-\rho_\varepsilon)^r\langle[\Theta,\Lambda]\omega,\omega\rangle e^{-\varphi}\\
&\lesssim& \frac{q|\log\delta|}{\delta^{2}} \int_{\Omega_\varepsilon}(-\rho_\varepsilon)\frac{|\partial\varphi|^2|f|^2}{|g|^{2q}} \cdot\int_{\Omega_\varepsilon}(-\rho_\varepsilon)^r\langle[\Theta,\Lambda]\omega,\omega\rangle e^{-\varphi} \ \ \ (\text{by\ }\eqref{eq:v-bound})\\
&\lesssim&  \frac{q|\log\delta|}{\delta^{2+2q}} \int_{\Omega_\varepsilon}(-\rho_\varepsilon)|f|^2\left|\partial(|z|^2+\lambda)\right|^2\cdot{Q(r,\varepsilon)}\\
& & +\frac{q|\log\delta|}{\delta^{2+2q-2}}\int_{\Omega_\varepsilon}(-\rho_\varepsilon)\frac{|f|^2}{|g|^{2}}\left(q+\frac{1}{|\log\delta|}\right)^2\frac{|\partial|g|\,|^2}{|g|^2} \cdot Q(r,\varepsilon)
\ \ \ (\text{by\ }\eqref{eq:KeyIneqInGeneral})\\
&\lesssim& \frac{q^3|\log\delta|^2}{\delta^{2+2q}} \|f\|_{H^2(\Omega)}^2\cdot Q(r,\varepsilon)\ \ \ (\text{by\ }\eqref{eq:G_2},\ \eqref{eq:ag}).
\end{eqnarray*}
It remains to estimate $I_2$.  Since
\begin{eqnarray*}
i\partial\bar{\partial}\log|g|^2&=&\frac{i\partial\bar\partial |g|^2}{|g|^2}+\frac{i\partial |g|^2 \wedge \bar\partial |g|^2}{|g|^4}\\
&=&\frac{i}{|g|^4}\left(-\sum_{j,l}g_j\bar g_l\partial g_l\wedge \bar\partial \bar g_j+|g|^2\sum_j\partial g_j\wedge \bar\partial \bar g_j\right)\\
&=&\frac{i}{|g|^4}\sum_{j<l}(g_j\partial g_l-g_l\partial g_j)\wedge(\overline{g_j\partial g_l-g_l\partial g_j}),
\end{eqnarray*}
it follows that
\begin{eqnarray}\label{eq:general_2}
\sum_{\mu,\nu=1}^n \frac{\partial^2 \log |g|^2}{\partial z_\mu \partial \bar{z}_\nu}\xi_\mu \bar{\xi}_\nu
&=& \frac{1}{|g|^4} \sum_{j<l} \left|\sum_{\mu=1}^n \left(g_j\frac{\partial g_l}{\partial z_\mu}-
g_l\frac{\partial g_j}{\partial z_\mu}\right) \xi_\mu \right|^2\nonumber\\
&=& \frac{1}{2|g|^4} \sum_{j,l} \left|\sum_{\mu=1}^n \left(g_j\frac{\partial g_l}{\partial z_\mu}-
g_l\frac{\partial g_j}{\partial z_\mu}\right) \xi_\mu \right|^2
\end{eqnarray}
holds for any vector $\xi=(\xi_1,\cdots,\xi_n)\in \mathbb C^n$.   In particular, we have
\begin{eqnarray}\label{eq:general_3}
\langle[\Theta,\Lambda]\omega,\omega\rangle
&\geq&\frac{1}{|\log\delta|}\sum_{k}\sum_{\mu,\nu}\frac{\partial^2\log |g|^2}{\partial z_\mu \partial\bar z_\nu}\omega_{k\mu}\bar{\omega}_{k\nu}\nonumber\\
&=& \frac{1}{2|g|^4|\log\delta|}\sum_{k,j,l} \left|\sum_{\mu=1}^n \left(g_j\frac{\partial g_l}{\partial z_\mu}-
g_l\frac{\partial g_j}{\partial z_\mu}\right) \omega_{k\mu} \right|^2\nonumber\\
&\geq& \frac{1}{2|g|^4|\log\delta|}\sum_{k,j} \left|\sum_{\mu=1}^n \left(g_j\frac{\partial g_k}{\partial z_\mu}-
g_k\frac{\partial g_j}{\partial z_\mu}\right) \omega_{k\mu} \right|^2
\end{eqnarray}
On the other hand,  we have
\begin{eqnarray*}
\frac{\partial v_{k\mu}}{\partial z_\nu}&=&  
\frac{\partial}{\partial z_\nu} \left[ \frac{f}{|g|^4}\sum_{j=1}^m  g_j\left(\bar g_j\frac{\partial \bar g_k}{\partial\bar z_\mu}-\bar g_k\frac{\partial \bar g_j}{\partial \bar z_\mu}\right)\right]\\
&=&\frac{\partial (f/|g|^4)}{\partial z_\nu}\cdot \sum_{j=1}^m  g_j\left(\bar g_j\frac{\partial \bar g_k}{\partial\bar z_\mu}-\bar g_k\frac{\partial \bar g_j}{\partial \bar z_\mu}\right)\\
&& +\frac{f}{|g|^4}\sum_{j=1}^m\frac{\partial g_j}{\partial z_\nu}\left(\bar g_j\frac{\partial \bar g_k}{\partial\bar z_\mu}-\bar g_k\frac{\partial \bar g_j}{\partial \bar z_\mu}\right).
\end{eqnarray*}
Thus
\begin{eqnarray*}
I_2&=&\sum_{k,j} \int_{\Omega_\varepsilon} (-\rho_\varepsilon)\sum_{\mu,\nu} \chi |\partial d|^{-2}\frac{\partial(-d)}{\partial\bar z_\nu}\frac{\partial(f/|g|^4)}{\partial z_\nu}g_j\left(\bar g_j\frac{\partial \bar g_k}{\partial\bar z_\mu}-\bar g_k\frac{\partial \bar g_j}{\partial \bar z_\mu}\right)\bar\omega_{k\mu}e^{-\varphi}\\
&&+\sum_{k,j} \int_{\Omega_\varepsilon} (-\rho_\varepsilon)\sum_{\mu,\nu}  \chi |\partial d|^{-2}\frac{\partial(-d)}{\partial\bar z_\nu}\frac{f}{|g|^4}\frac{\partial g_j}{\partial z_\nu}\left(\bar g_j\frac{\partial \bar g_k}{\partial\bar z_\mu}-\bar g_k\frac{\partial \bar g_j}{\partial \bar z_\mu}\right)\bar\omega_{k\mu}e^{-\varphi}\\
&=:&J_1+J_2.
\end{eqnarray*}
By Cauchy-Schwarz inequality, we have
\begin{eqnarray*}
|J_1|^2&\lesssim& |\log\delta|\int_{\Omega_\varepsilon} (-\rho_\varepsilon) |g|^6 |\partial(f/|g|^4)|^2 e^{-\varphi}\\
&&\cdot  \int_{\Omega_\varepsilon} (-\rho)^r\frac{1}{|g|^4|\log\delta|}\sum_{j,k} \left|\sum_{\mu=1}^n \left(g_j\frac{\partial g_k}{\partial z_\mu}-
g_k\frac{\partial g_j}{\partial z_\mu}\right) \omega_{k\mu} \right|^2e^{-\varphi}\\
&\lesssim& |\log\delta|\int_{\Omega_\varepsilon} (-\rho_\varepsilon)\left[\frac{|\partial f|^2}{|g|^{2+2q}}+\frac{|f|^2 |\partial |g|\,|^2}{|g|^{4+2q}}\right]\cdot Q(r,\varepsilon)\ \ \ (\text{by\ }\eqref{eq:general_3},\eqref{eq:KeyIneqInGeneral})\\
&\lesssim& \frac{|\log\delta|^2}{\delta^{2+2q}}\|f\|_{H^2(\Omega)}^2\cdot Q(r,\varepsilon) \ \ \ (\text{by\ }\eqref{eq:G_3},\eqref{eq:ag}).
\end{eqnarray*}
Analogously,  we have
\begin{eqnarray*}
|J_2|^2&\lesssim& |\log\delta|\int_{\Omega_\varepsilon} (-\rho_\varepsilon)\frac{|f|^2}{|g|^4}|\partial g|^2 e^{-\varphi}\\
&&\cdot  \int_{\Omega_\varepsilon} (-\rho)^r\frac{1}{|g|^4|\log\delta|}\sum_{j,k} \left|\sum_{\mu=1}^n \left(g_j\frac{\partial g_k}{\partial z_\mu}-
g_k\frac{\partial g_j}{\partial z_\mu}\right) \omega_{k\mu} \right|^2e^{-\varphi}\\
&\lesssim& |\log\delta|\int_{\Omega_\varepsilon} (-\rho_\varepsilon)\frac{|f|^2}{|g|^4}|\partial g|^2 e^{-\varphi} \cdot Q(r,\varepsilon)\\
&\lesssim&\frac{|\log\delta|}{\delta^{2q}}\int_{\Omega_\varepsilon} (-\rho_\varepsilon) \frac{|f|^2}{|g|^4}|\partial g|^2 \cdot Q(r,\varepsilon)\\
&\lesssim&\frac{|\log\delta|^2}{\delta^{2+2q}}\|f\|_{H^2(\Omega)}^2\cdot Q(r,\varepsilon)
\end{eqnarray*}
in view of \eqref{eq:g}.  
These together with \eqref{eq:2generator_1} and \eqref{eq:2generator_2} yield
$$
\left|\int_{\Omega_\varepsilon} \langle v,\omega\rangle e^{-\varphi}\right|^2 \lesssim  \frac{|\log\delta|^{2}}{\delta^{2+2q}}\|f\|^2_{H^2(\Omega)} \cdot \frac1{1-r}\int_{\Omega_\varepsilon} |T^\ast_{\mathcal{H},\varphi}\omega|^2 (-\rho)^{r} e^{-\varphi},
$$    
for any $\omega\in {\rm Ker\,}S_\mathcal H\cap \mathcal{D}_1(\Omega_\varepsilon)$. 
In view of Lemma \ref{lm:density},  the same inequality remains valid for all $\omega\in {\rm Ker\,}S_\mathcal H\cap \mathrm{Dom} (T^\ast_{\mathcal{H},\varphi})$. 
Repeat the duality argument in the case $m=2$ with  $T^\ast_\varphi$ replaced by $T^\ast_{\mathcal{H},\varphi}$,  we get a solution $u_\varepsilon$ of \eqref{eq:dbar} which satisfies the following estimate 
$$
          (1-r) \int_{\Omega_\varepsilon} |u_\varepsilon|^2 (-\rho)^{-r} e^{-\varphi} \leq C(\delta)^2  \|f\|^2_{H^2(\Omega)},\ \ \ \forall\,0<r<1,
  $$
  from which \eqref{eq:general_1} immediately follows since $\varphi\le \sup_{z\in \Omega}\{|z|^2+\lambda(z)\}$.
The proof of Theorem \ref{th:Main}
 is complete.
\section{A heuristic remark on the three-weights technique}
In this section we explain briefly why the three-weights technique probably fails for solving the $H^2$ corona problem (in the case $m=2$).   Let $\varphi_1=\varphi-2\psi$,  $\varphi_2=\varphi-\psi$ and $\varphi_3=\varphi$,  where $\varphi$ is a smooth psh function $\Omega_\varepsilon$ and $\psi$ is a smooth real-valued function on $\Omega_\varepsilon$.  The $\bar{\partial}-$operator induces two closed,  densely defined operators
$$
T: L^2(\Omega_\varepsilon,\varphi_1)\rightarrow L^2_{(0,1)}(\Omega_\varepsilon,\varphi_2),\ \ \ S: L^2_{(0,1)}(\Omega_\varepsilon,\varphi_2)\rightarrow L^2_{(0,2)}(\Omega_\varepsilon,\varphi_3).
$$
Let $\mathscr{D}_{(0,1)}(\Omega_\varepsilon)$ be the set of smooth $(0,1)-$forms with compact supports in $\Omega_\varepsilon$ and $T^\ast$ the adjoint of $T$.   It is known from \cite{HormanderBook} that one can choose $\psi$ such that $\mathscr{D}_{(0,1)}(\Omega_\varepsilon)$ lies dense in $\mathrm{Dom}(T^\ast)\cap \mathrm{Dom}(S)$ under the graph norm
$$
\omega \mapsto \|\omega\|_{\varphi_2}+\|T^\ast \omega\|_{\varphi_1} + \|S\omega \|_{\varphi_3}
$$
and 
\begin{eqnarray}\label{eq:three-weights_1}
(1+\delta)\|T^\ast \omega\|_{\varphi_1}^2 + \|S\omega\|_{\varphi_3}^2
& \ge & \int_{\Omega_\varepsilon} \left(\sum_{j,k=1}^n \frac{\partial^2\varphi}{\partial z_j\partial \bar{z}_k}\omega_j\bar{\omega}_k-(1+1/\delta)|\partial \psi|^2 |\omega|^2\right) e^{-\varphi}\nonumber\\
& & + \int_{\Omega_\varepsilon} \sum_{j,k}\left|\frac{\partial \omega_j}{\partial \bar{z}_k}\right|^2 e^{-\varphi}
\end{eqnarray}
holds for any $\omega\in \mathscr{D}_{(0,1)}(\Omega_\varepsilon)$ and $\delta>0$.
Replace $\varphi$ by $\varphi':=\varphi + \kappa\circ s$ in \eqref{eq:three-weights_1},   where $s$ is a strictly psh exhaustion function on $\Omega_\varepsilon$ and $\kappa$ is certain convex  (rapidly) increasing function (depending on $\delta$),
we obtain
\begin{eqnarray}\label{eq:three-weights_2}
(1+\delta)\|T^\ast \omega\|_{\varphi_1}^2 + \|S\omega\|_{\varphi_3}^2 
&\ge& \int_{\Omega_\varepsilon} \sum_{j,k=1}^n \left(\frac{\partial^2\varphi}{\partial z_j\partial \bar{z}_k}
+\frac12  \frac{\partial^2\kappa\circ s}{\partial z_j\partial \bar{z}_k}\right)
\omega_j\bar{\omega}_k e^{-\varphi'}\nonumber\\
 && + \int_{\Omega_\varepsilon} \sum_{j,k}\left|\frac{\partial \omega_j}{\partial \bar{z}_k}\right|^2 e^{-\varphi'}.
\end{eqnarray}
We remark that $\delta$ will tends to zero in the end. 

A similar argument as \S\,3 yields
\begin{eqnarray}\label{eq:three-weights_3}
&& \int_{\Omega_\varepsilon} |S \omega |^2 (-\rho)^{r} e^{-\varphi_3} + \frac{1+\delta}{1-(1+\delta)r}
 \int_{\Omega_\varepsilon} |T^\ast \omega|^2 (-\rho)^{r} e^{-\varphi_1}\\
 & \ge & \int_{\Omega_\varepsilon} \sum_{j,k} \left[\frac{\partial^2\varphi}{\partial z_j\partial\bar{z}_k} +\frac12  \frac{\partial^2\kappa\circ s}{\partial z_j\partial \bar{z}_k} +r (-\rho)^{-1}
 \frac{\partial^2\rho}{\partial z_j\partial\bar{z}_k}\right] \omega_j\bar{\omega}_k (-\rho)^{r} e^{-\varphi'} \nonumber\\
 && + \int_{\Omega_\varepsilon} \sum_{j,k}\left|\frac{\partial \omega_j}{\partial \bar{z}_k}\right|^2 (-\rho)^{r} e^{-\varphi'},\ \ \ \forall\,\omega\in \mathscr{D}_{(0,1)}(\Omega_\varepsilon).\nonumber
  \end{eqnarray}
  Substitute $\varphi(z)=\lambda(z)+|z|^2+\log |g(z)|^2$ into \eqref{eq:three-weights_3},  we obtain
\begin{eqnarray}\label{eq:three-weights_4}
&& \int_{\Omega_\varepsilon} |S \omega|^2 (-\rho)^{r} e^{-\varphi_3} + \frac{1+\delta}{1-(1+\delta)r}
 \int_{\Omega_\varepsilon} |T^\ast \omega|^2 (-\rho)^{r} e^{-\varphi_1}\\
 & \ge & \int_{\Omega_\varepsilon} \sum_{j,k} \Theta_{j\bar{k}} \omega_j\bar{\omega}_k (-\rho)^{r} e^{-\varphi'} 
   + \int_{\Omega_\varepsilon} \sum_{j,k}\left|\frac{\partial \omega_j}{\partial \bar{z}_k}\right|^2 (-\rho)^{r} e^{-\varphi'}\nonumber,
  \end{eqnarray}   
  where $\Theta_{j\bar{k}}=\delta_{jk}+\partial^2 \log |g|^2/\partial z_j\partial \bar{z}_k + \frac12 \partial^2 \kappa\circ s/\partial z_j\partial \bar{z}_k$.
  
 Again,  we have the following decomposition 
$$
\int_{\Omega_\varepsilon} v\cdot \bar{\omega} e^{-\varphi_2}=I_1+I_2+I_3+I_4,
$$
where 
$$
I_4=-\int_{\Omega_\varepsilon}(-\rho_\varepsilon) \sum_j \chi|\partial d|^{-2} \frac{\partial (-d)}{\partial \bar{z}_j} \frac{\partial \varphi_2}{\partial z_j} v\cdot \bar{\omega} e^{-\varphi_2}.
$$
Note that
$$
|I_4|^2 \lesssim  \int_{\Omega_\varepsilon}(-\rho_\varepsilon)^{2-r}|\partial\varphi_2|^2 |v|_\Theta^2  e^{-\varphi_1}\cdot\int_{\Omega_\varepsilon}(-\rho_\varepsilon)^r\langle[\Theta,\Lambda]\omega,\omega\rangle e^{-\varphi'}.
$$
Clearly,  $\varphi_1>0$ provided that $\kappa$ is so rapidly increasing,  i.e.,  $e^{-\varphi_1}<1$.  
On the other hand,  for $\partial \varphi_2=\partial \varphi + \kappa'(s)\partial s -\partial \psi$,  the term
$|\partial \varphi-\partial \psi|^2 |v|_\Theta^2$ may be controlled  provided that $\kappa$ is so rapidly increasing,  while the term $\kappa'(s)^2 |\partial s|^2 |v|_\Theta^2$ is controllable unless at least
$$
i\partial\bar{\partial} \kappa \circ s = i \kappa''(s)\partial\bar{\partial} s + i\kappa'(s) \partial s\wedge \bar{\partial} s \gtrsim i\kappa'(s)^2 \partial s\wedge \bar{\partial} s,
$$
but this seems impossible!

{\bf Acknowledgements.} The authors would like to thank the referee for valuables comments and   Yuanpu Xiong for suggesting a proof of Lemma \ref{lm:density}.
 

\begin{thebibliography}{99}  
 \bibitem{Alling} N.  L.  Alling,  {\it A proof of the corona conjecture for finite open Riemann surfaces},  Bull.  Amer.  Math.  Soc.  {\bf 70} (1964), 110--112.
 \bibitem{Amar} E.  Amar, {\it On the corona problem},  J. Geom. Anal.  {\bf 1} (1991),  291--305.  
\bibitem{AnderssonBook} M.  Andersson, Topics in Complex Analysis,  Universitext,  Springer-Verlag,  New York, 1997.
   \bibitem{Andersson} M. Andersson, {\it The $H^2$ corona problem and $\bar{\partial}_b$ in weakly pseudoconvex domains}, Trans. Amer. Math. Soc. {\bf 342} (1994), 241--255.
   \bibitem{Andersson2} M.  Andersson,  { \it  On the $H^p$ corona problem}, Bull. Sci. Math. {\bf 118}  (1994),  287--306.
   \bibitem{AC} M. Andersson and H. Carlsson,  {\it Estimates of solutions of the $H^p$ and BMO corona problem},  Math.  Ann.  {\bf 316} (2000),  83--102.
   \bibitem{Berndtsson} B. Berndtsson, {\it $\bar{\partial}_b$ and Carleson type inequalities}, In: Complex Analysis. II. Proceedings of the special year held at the University of Maryland, College Park, Md., July 1985--December 1986. (C. A. Berenstein eds.) Lecture Notes in Math., {\bf 1276}, Springer-Verlag, Berlin, (1987), 42--54.
   \bibitem{Carleson} L. Carleson, {\it Interpolations by bounded analytic functions and the corona problem}, Ann. of Math. {\bf 76} (1962), 547--559.
   \bibitem{Catlin84} D.  Catlin,  {\it Global regularity of the $\bar{\partial}-$Neumann problem}, Proc.  Sympos.  Pure Math.,  {\bf 41},  Amer. Math. Soc,  Providence,  R.I., 1984,   39--49.
 \bibitem{Catlin87}  D.  Catlin,  {\it Subelliptic estimates for the $\bar{\partial}-$Neumann problem on pseudoconvex domains},  Ann.  of Math.  {\bf 126} (1987), 131--191.
   \bibitem{ChenFu} B. Y. Chen and S. Fu, {\it Comparison of the Bergman and Szeg\"o kernels}, Adv. Math. {\bf 228} (2011), 2366--2384.
    \bibitem{Chen14} B. Y. Chen, {\it Weighted Bergman spaces and the $\bar{\partial}-$equation}, Trans. Amer. Math. Soc. {\bf 366} (2014), 4127--4150.
  \bibitem{DAngelo}  J.  P.  D'Angelo,  {\it Real hypersurfaces. orders of contact, and applications},   Ann.  of Math.  {\bf 115}
 (1982),  615--637.
    \bibitem{DemaillyBook} J.-P.  Demailly,  Complex Analytic and Differential Geometry,  lecture notes, 2012,  available at \url{https://www-fourier.ujf-grenoble.fr/~demailly}.
    \bibitem{DF1} K.  Diederich and J.  E.  Forn{\ae}ss,  {\it Pseudoconvex domains: Bounded strictly plurisubharmonic exhaustion functions}, Invent. Math. {\bf 39} (1977), no. 2, 129--141.
 \bibitem{Douglas} R.  G.  Douglas et al.,  {\it A history of the Corona problem},  In: The Corona Problem:  Connections Between Operator Theory,  Function Theory,  and Geometry (R.  G.  Douglas et al.  eds.),  Fields Institute Communications {\bf 72},  1--29,  Springer 2014.
 \bibitem{Fornaess-Sibony} J.  E.  Forn\ae ss and N. Sibony, {\it Smooth pseudoconvex domains in $\mathbb{C}^2$ for which the Corona theorem and $L^p$ Estimates for $\dbar$ fail}, In: Complex analysis and Geometry (V. Ancona and A. Silva eds.), Plenum Press, New York, (1993), 209--222. 
\bibitem{Garnnett} J. Garnett, {Bounded Analytic Functions}, Academic Press, 1981.
   \bibitem{Hormander} L. H\"ormander, {\it $L^2$ estimates and existence theorems for the $\bar{\partial}-$operator}, Acta Math. {\bf 113} (1965), 89--152.
   \bibitem{Hormander67} L. H\"ormander,  {\it Generators for some rings of analytic functions},  Bull.  Amer.  Math.  Soc.  {\bf 73} (1967),  943--949.  
 \bibitem{HormanderBook} L. H\"ormander, An Introduction to Complex Analysis in Several Variables, Third Edition (Revised), Elsevier, 1990.
 \bibitem{KZ} T. V. Khanh and G. Zampieri, {\it Regularity of the $\bar\partial-$Neumann problem at point of infinite type}, J. Funct. Anal. {\bf 259} (2010), 2760--2775.
 \bibitem{Kohn} J. J. Kohn, {\it Superlogarithmic estimates on pseudoconvex domains and CR manifolds}, Ann. of Math. {\bf 156} (2002), 213--248.
\bibitem{Ohsawa} T. Ohsawa, Analysis of several complex variables, Translations of Mathematical Monographs {\bf 211}, Ametican Mathematical Society, Providence, RI, 2002.
   \bibitem{OT} T. Ohsawa and K. Takegoshi, {\it On the extension of $L^2$ holomorphic functions}, Math. Z. {\bf 195} (1987), 197--204.
   \bibitem{Sibony72} N.  Sibony,  {\it Prolongement analytique des fonctions holomorphes born\'ees},  C.  R.  Acad.  Sci.  Paris Ser.  A--B {\bf 275} (1972),  973--976.  
  \bibitem{Sibony87}  N.  Sibony,  {\it Probl\'eme de la couronne pour des domaines pseudoconvexes bord lisse},  Ann.  of Math.  {\bf 126} (1987),  675--682.
   \bibitem{Skoda} H. Skoda, {\it Application des techniques $L^2$ \`{a} la th\'{e}orie des id\'{e}aux d'une alg\`{e}bre de fonctions holomorphes avec poids}, Ann. Sci. \'{E}cole Norm. Sup. {\bf 5} (1972), 545--579.
   \bibitem{Stein} E. M. Stein,  Boundary Behavior of Holomorphic Functions of Several Complex Variables,  Princeton University Press,  Princeton,  NJ; University of Tokyo Press, Tokyo,  1972.   
   \bibitem{Stout} E.  L.  Stout,  {\it Bounded holomorphic functions on finite Riemann surfaces}, Trans.  Amer.  Math.  Soc.  {\bf 120} (1965),  255--285.
   \bibitem{Tolokonnikov} V. A. Tolokonnikov, {\it Estimates in the Carleson corona theorem, ideals of the algebra $H^\infty$, a problem of Sz.-Nagy}, J. Soviet Math. {\bf 22} (1983), 1814--1828.
   \bibitem{Treil} Treil, S. {\it Estimates in the corona theorem and ideals of $H^\infty$: A problem of T. Wolff}, J. Anal. Math. {\bf 87} (2002), 481--495.  
   \bibitem{Uchiyama}  A. Uchiyama, {\it Corona theorems for countably many functions and estimates for their solutions}, preprint (1980).
 \bibitem{Varopoulos} N. Varopoulos,  {\it BMO functions and the $\bar{\partial}-$equation},  Pacific J.  Math. {\bf 71} (1977),  221--273.
    \end{thebibliography}
          \end{document}